\documentclass[a4paper,10pt]{article}
\usepackage{amssymb,amsmath,amsthm}
\usepackage{amsfonts}
\usepackage[english]{babel}
\usepackage{hyperref}  

\numberwithin{equation}{section}

\newtheorem{theorem}{Theorem}[section]
\newtheorem{proposition}[theorem]{Proposition}
\newtheorem{lemma}[theorem]{Lemma}
\newtheorem{corollary}[theorem]{Corollary}

\newtheorem{remark}[theorem]{Remark}
\newtheorem{remarks}[theorem]{Remark}
\newtheorem{definition}[theorem]{Definition}
\newcommand{\be}{\begin{equation}}
\newcommand{\ee}{\end{equation}}

\newcommand{\e}{\varepsilon}

\newcommand{\R}{\mathbb R}
\newcommand{\C}{\mathbb C}

\newcommand{\Z}{\mathbb Z}

\newcommand{\N}{\mathbb N}
\newcommand{\T}{\mathbb T}

\renewcommand{\b }{\beta }
\newcommand{\s }{\sigma }
\newcommand{\ii }{{\rm i} }

\newcommand{\vphi}{\varphi }

\newcommand{\pa}{\partial}

\def\ba{\begin{aligned}}
\def\ea{\end{aligned}}
\def\beginm{\begin{multline}}
\def\endm{\end{multline}}

\begin{document}

\title{{\bf On the growth of Sobolev norms for a class of linear Schr\"odinger equations on the torus with superlinear dispersion}}

\date{}

 \author{Riccardo Montalto \footnote{Supported in part by the Swiss National Science Foundation}}

\maketitle

\noindent
{\bf Abstract:}
In this paper we consider time dependent Schr\"odinger equations on the one-dimensional torus $\T := \R /(2 \pi \Z)$ of the form $\partial_t u = \ii {\cal V}(t)[u]$ where ${\cal V}(t)$ is a time dependent, self-adjoint pseudo-differential operator of the form ${\cal V}(t) = V(t, x) |D|^M + {\cal W}(t)$, $M > 1$, $|D| := \sqrt{- \partial_{xx}}$, $V$ is a smooth function uniformly bounded from below and ${\cal W}$ is a time-dependent pseudo-differential operator of order strictly smaller than $M$. We prove that the solutions of the Schr\"odinger equation $\partial_t u = \ii {\cal V}(t)[u]$ grow at most as $t^\e$, $t \to + \infty$ for any $\e > 0$. The proof is based on a reduction to constant coefficients up to smoothing remainders of the vector field $\ii {\cal V}(t)$ which uses Egorov type theorems and pseudo-differential calculus. 
 \\[2mm]
\\[1mm]

\tableofcontents

\section{Introduction and main result}
In this paper we consider linear Schr\"odinger equations of the form
\begin{equation}\label{main equation}
\partial_t u + \ii {\cal V}(t)[u] = 0 \,, \quad x \in \T
\end{equation}
where $\T := \R / (2 \pi \Z)$ is the $1$-dimensional torus, ${\cal V}(t)$ is a $L^2$ self-adjoint, time dependent, pseudo-differential  Schr\"odinger operator of the form 
\begin{equation}\label{forma iniziale cal V (t)}
{\cal V}(t) := V(t, x) |D|^M + {\cal W}(t)\,, \quad |D| := \sqrt{- \partial_{xx}}\,,\quad M > 1\,.
\end{equation}
We assume that $ V$ is a real valued ${\cal C}^\infty$ function defined on $\R \times \T$ with all derivatives bounded satisfying $\inf_{(t, x) \in \R \times \T} V(t, x) > 0$ and ${\cal W}(t)$ is a time-dependent pseudo differential operator of order strictly smaller than $M$. Our main goal is to show that given $t_0 \in \R$, $s \geq 0$, $u_0 \in H^s(\T)$, the Cauchy problem 
 \begin{equation}
 \begin{cases}
 \partial_t u + \ii {\cal V}(t)[u] = 0  \\
 u(t_0, x) = u_0(x)
 \end{cases}
 \end{equation}  
 admits a unique solution $u(t)$ satisfying, for any $\e > 0$, the bound $\| u(t)\|_{H^s} \leq C(s, \e) (1 + |t - t_0|)^\e \| u_0\|_{H^s}$ for some constant $C(s, \e) > 0$. Here, $H^s(\T)$ denotes the standard Sobolev space on the 1-dimensional torus $\T$ equipped with the norm $\| \cdot \|_{H^s}$. 
 
 \noindent
 There is a wide literature concerning the problem of estimating the high Sobolev norms of linear partial differential equations. For the Schr\"odinger operator ${\cal V}(t) = - \Delta + V(t, x)$ on the $d$-dimensional torus $\T^d$, the growth $\sim t^\e$ of the $\| \cdot\|_{H^s}$ norm of the solutions of $\partial_t u = \ii {\cal V}(t)[u]$ has been proved by Bourgain in \cite{bourgain-1} for smooth quasi-periodic in time potentials and in \cite{bourgain-2} for smooth and bounded time dependent potentials. In the case where the potential $V$ is analytic and quasi-periodic in time, Bourgain  \cite{bourgain-1} proved also that $\| u(t) \|_{H^s}$ grows like a power of $\log(t)$. Moreover, this bound is optimal, in the sense that he constructed an example for which $\| u(t)\|_{H^s}$ is bounded from below by a power of $\log(t)$. The result obtained in \cite{bourgain-2} has been extended by Delort \cite{delort} for Schr\"odinger operators on Zoll manifolds. Furthermore, the logarithmic growth of $\| u(t) \|_{H^s}$ proved in \cite{bourgain-1} has been extended by Wang \cite{wang-1} in dimension $1$, for any real analytic and bounded potential. The key idea in these series of papers is to use the so-called {\it spectral gap condition} for the operator $- \Delta$. Such a condition states that the spectrum of $- \Delta$ can be enclosed in disjont clusters $(\sigma_j)_{j \geq 0}$ such that the distance between $\sigma_j$ and $\sigma_{j + 1}$ tends to $+ \infty$ for $j \to + \infty$.  
 
 \noindent
 All the aforementioned results deal with the Schr\"odinger operator with a multiplicative potential. The first result in which the growth of $\| u(t) \|_{H^s}$ is exploited for Schr\"odinger operators with unbounded perturbations is due to Maspero-Robert \cite{albertino}. More precisely, they prove the growth $\sim t^\e$ of $\| u(t)\|_{H^s}$, for Schr\"odinger equations of the form $\ii \partial_t u = L(t) u$ where $L(t) = H + P(t)$, $H$ is a time-independent operator of order $\mu  + 1$ satisfying the {\it spectral gap condition} and  $P(t)$ is an operator of order $\nu \leq \mu/(\mu + 1)$ (see Theorem 1.8 in \cite{albertino}).   
 The purpose of this paper is to provide a generalization of the result obtained in \cite{albertino}, at least for Schr\"odinger operators on the $1$-dimensional torus, when the order of $H$ is the same as the order of $P(t)$. Note that the operator defined in \eqref{forma iniziale cal V (t)} can be written in the form $H + P(t)$ where $H = |D|^M$, $P(t) = (V(t, x) - 1) |D|^M + {\cal W}(t)$ and the operator $|D|^M$ fullfills the {\it spectral gap condition} since $M > 1$ (superlinear growth of the eigenvalues). Another generalization of \cite{albertino} has been obtained independently and at the same time as our paper by Bambusi-Grebert-Maspero-Robert \cite{albertino2} in the case in which the order of $P(t)$ is strictly smaller than the one of $H$. This result covers also several applications in higher space dimension.
 
 \noindent
 We also mention that in the case of quasi-periodic systems $\ii \partial_t u = L(\omega t)[u] $, $L(\omega t) = H + \e P(\omega t)$ it is often possible to prove that $\| u(t) \|_{H^s}$ is uniformly bounded in time for $\e$ small enough and for a {\it large }set of frequencies $\omega$. The general strategy to deal with these quasi-periodic systems is called {\it reducibility}. It consists in costructing, for most values of the frequencies $\omega$ and for $\e$ small enough, a bounded quasi-periodic change of variable $\Phi(\omega t)$ which transforms the equation $\ii \partial_t u = L(\omega t) u$ into a time independent system $\ii \partial_t v = {\cal D} v$ whose solution preserves the Sobolev norms $\| v(t) \|_{H^s}$. We mention the results of Eliasson-Kuksin \cite{EK1} which proved the reducibility of the Schr\"odinger equation on $\T^d$ with a small, quasi-periodic in time analytic potential and Grebert-Paturel \cite{GrebertPaturel} which proved the reducibility of the quantum harmonic oscillator on $\R^d$. Concerning KAM-reducibility with unbounded perturbations, we mention Bambusi \cite{Bambusi1}, \cite{Bambusi2} for the reducibility of the quantum harmonic oscillator with unbounded perturbations (see also \cite{albertino3} in any dimension), \cite{BBM-Airy}, \cite{BBM-auto}, \cite{Giuliani} for fully non-linear KdV-type equations, \cite{Feola}, \cite{Feola-Procesi} for fully-nonlinear Schr\"odinger equations, \cite{BM16}, \cite{BertiMontalto} for the water waves system and \cite{Montalto} for the Kirchhoff equation. Note that in \cite{BBM-Airy}, \cite{BBM-auto}, \cite{Giuliani}, \cite{BM16}, \cite{BertiMontalto}, \cite{Montalto} the reducibility of the linearized equations is obtained as a consequence of the KAM theorems proved for the corresponding nonlinear equations. 
 
 \noindent
 We now state in a precise way the main results of this paper. First, we introduce some notations. For any function $u \in L^2(\T)$, we introduce its Fourier coefficients 
\begin{equation}\label{coefficienti fourier}
\widehat u(\xi) := \frac{1}{2 \pi} \int_\T u(x) e^{- \ii x \xi}\, d x\,, \qquad \forall \xi \in \Z\,. 
\end{equation}
For any $s \geq 0$, we introduce the Sobolev space of complex valued functions $H^s \equiv H^s(\T)$, as 
\begin{equation}\label{definizione sobolev}
\begin{aligned}
& H^s  := \Big\{ u \in L^2(\T) :  \| u \|_{H^s}^2 :=\sum_{\xi \in \Z} \langle \xi \rangle^{2 s} |\widehat u(\xi) |^2 < + \infty \Big\}\,,\quad  \langle \xi \rangle := (1 + |\xi|^2)^{\frac12}\,. 
\end{aligned}
\end{equation}

\noindent
Given two Banach spaces $(X, \| \cdot \|_X)$, $(Y, \| \cdot \|_Y)$, we denote by ${\cal B}(X, Y)$ the space of bounded linear operators from $X$ to $Y$ equipped with the usual operator norm $\| \cdot \|_{{\cal B}(X, Y)}$. If $X = Y$, we simply write ${\cal B}(X)$ for ${\cal B}(X, X)$. 

\noindent
Given a linear operator ${\cal R} \in {\cal B}(L^2(\T))$, we denote by ${\cal R}^*$ the adjoint operator of ${\cal R}$ with respect to the standard $L^2$ inner product 
\begin{equation}
\langle u, v \rangle_{L^2} := \int_\T u(x) \overline{v(x)}\, d x\,, \quad \forall u, v \in L^2(\T)\,. 
\end{equation} 
 We say that the operator ${\cal R}$ is self-adjoint if ${\cal R} = {\cal R}^*$. 

\noindent
Given a Banach space $(X, \| \cdot \|_X)$, for any $k \in \N$, for any $- \infty \leq T_1 < T_2 \leq + \infty$ we consider the space ${\cal C}^k([T_1, T_2], X)$ of the $k$-times continuously differentiable functions with values in $X$. We denote by ${\cal C}^k_b([T_1, T_2], X)$ the space of functions in ${\cal C}^k([T_1, T_2], X)$ having bounded derivatives, equipped with the norm 
\begin{equation}\label{norme Ck tempo}
\| u \|_{{\cal C}^k_b([T_1, T_2], X)} := {\rm max}_{j =1, \ldots, k} \sup_{t \in [T_1, T_2]} \| \partial_t^j u(t) \|_X\,.  
\end{equation}
For any domain $\Omega \subset \R^d$, we also denote by ${\cal C}^\infty_b(\Omega)$ the space of the ${\cal C}^\infty$ functions on $\Omega$ with all the derivatives bounded. 

\noindent
Since the equation we deal with is a Hamiltonian PDE, we briefly describe the Hamiltonian formalism. We define the symplectic form $\Omega : L^2(\T ) \times L^2(\T ) \to \R$ by 
\begin{equation}\label{forma simplettica coordinate complesse}
\Omega[u_1, u_2] : =  \ii \int_{\T} (u_1 \bar u_2 -  \bar u_1 u_2)\, dx\,, \quad \forall u_1, u_2 \in L^2(\T)\,. 
\end{equation}
Given a family of linear operators ${\cal R} : \R \to {\cal B}(L^2)$ such that ${\cal R}(t) = {\cal R}(t)^*$ for any $t \in \R$, we define the time-dependent quadratic Hamiltonian associated to ${\cal R}$ as 
$$
{\cal H}(t, u) := \langle {\cal R}(t)[u]\,,\, u \rangle_{L^2_x} = \int_\T {\cal R}(t)[u]\, \overline u\, d x\,, \qquad \forall u \in L^2(\T)\,. 
$$
The Hamiltonian vector field associated to the Hamiltonian ${\cal H}$ is defined by
\begin{equation}\label{campo vettoriale hamiltoniano}
X_{\cal H}(t, u) := \ii \nabla_{\overline u} {\cal H}(t, u) = \ii {\cal R}(t)
\end{equation}
where the gradient $\nabla_{\overline u}$ stands for 
$$
\nabla_{\overline u} := \frac{1}{\sqrt{2}} (\nabla_v + \ii \nabla_\psi)\,, \quad v = {\rm Re}(u)\,, \quad \psi := {\rm Im}(u)\,. 
$$
We say that $\Phi : \R \to {\cal B}(L^2(\T))$ is symplectic if and only if 
$$
\Omega\Big[ \Phi(t)[u_1], \Phi(t)[u_2] \Big] = \Omega [u_1, u_2 ]\,, \qquad \forall u_1, u_2 \in L^2(\T)\,, \quad \forall t \in \R\,.  
$$
We recall the classical thing that if $X_{\cal H}$ is a Hamiltonian vector field, then ${\rm exp}(X_{\cal H})$ is symplectic.

Let us consider a time dependent vector field $X : \R \to {\cal B}(L^2(\T))$ and a differentiable family of invertible maps $\Phi : \R \to {\cal B}(L^2(\T))$. Under the change of variables $u = \Phi(t)[v]$, the equation $\partial_t u = X(t)[u]$
transforms into the equation $\partial_t v = X_+(t)[v]$
where the {\it push-forward} $X_+(t)$ of the vector field $X(t)$ is defined by 
\begin{equation}\label{push forward}
X_+(t) := \Phi_* X(t) := \Phi(t)^{- 1} \Big( X(t) \Phi(t)  - \partial_t \Phi(t) \Big)\,, \quad t \in \R\,. 
\end{equation}
It is well known that if $\Phi$ is symplectic and $X(t)$ is a Hamiltonian vector field, then the push-forward $X_+(t) = \Phi_* X (t)$ is still a Hamiltonian vector field.

\noindent
In the next two definitions, we also define time dependent pseudo differential operators on $\T$.
\begin{definition}[{\bf The symbol class} $S^m$]\label{classe simboli}
Let $m \in \R$. We say that a ${\cal C}^\infty$ function $a :  \R \times \T \times \R \to \C$ belongs to the symbol class $S^m$ if and only if for any $\alpha, \beta, \gamma \in \N$ there exists a constant $C_{\alpha, \beta, \gamma} > 0$ such that 
\begin{equation}
|\partial_t^\alpha\partial_x^\beta  \partial_\xi^\gamma a( t, x, \xi)| \leq C_{\alpha, \beta, \gamma} \langle \xi \rangle^{m - \gamma}\,, \quad \forall (t, x, \xi) \in  \R \times \T \times \R\,. 
\end{equation}
We define the class of smoothing symbols $S^{- \infty} := \cap_{m \in \R} S^m$. 
\end{definition}
\begin{definition} \label{def:Ps2} {\bf (the class of operators $OPS^m$)}
Let $m \in \R$ and $a \in S^m$. We define the time-dependent linear operator $A(t) = {\rm Op}\big(a(t, x, \xi) \big) = a(t, x, D)$ as 
$$
A(t) [u] (x) := \sum_{\xi \in \Z} a(t,  x, \xi) \widehat u(\xi) e^{\ii x \xi}\,, \qquad \forall u \in {\cal C}^\infty(\T)\,.
$$
We say that the operator $A$ is in the class $OPS^m$. 

\noindent
We define the class of smoothing operators $OPS^{- \infty} := \cap_{m \in \R} OPS^m$. 
\end{definition}
Now, we are ready to state the main results of this paper. We make the following assumptions. 
  \begin{itemize}
  \item[\bf (H1)] The operator ${\cal V}(t) = V(t, x) |D|^M + {\cal W}(t)$ in \eqref{forma iniziale cal V (t)} is $L^2$ self-adjoint for any $t \in \R$.  
\item[\bf (H2)] The function $V(t, x)$ in \eqref{forma iniziale cal V (t)} is in $ {\cal C}^\infty_b(\R \times \T, \R)$, strictly positive and bounded from below, i.e. $\delta := \inf_{(t, x) \in \R \times \T} V(t, x) > 0$. 
\item[\bf (H3)] The operator ${\cal W}(t)$ is a time-dependent pseudo-differential operator ${\cal W}(t) = {\rm Op}(w(t, x, \xi))$, with symbol $w \in S^{M - \frak e}$ for some $\frak e > 0$. 

\end{itemize}

The main result of this paper is the following 
\begin{theorem}[\bf Growth of Sobolev norms]\label{teo growth of sobolev norms}
Assume the hypotheses {\bf (H1)}-{\bf (H3)}. Let $s > 0$, $u_0 \in H^s(\T)$, $t_0 \in \R$. Then there exists a unique global solution $u \in {\cal C}^0(\R, H^s(\T))$ of the Cauchy problem 
\begin{equation}\label{cauchy problem main theorem}
\begin{cases}
\partial_t u + \ii {\cal V}(t)[u] = 0 \\
u(t_0, x) = u_0(x)
\end{cases}
\end{equation}
and for any $\e > 0$ there exists a constant $C(s, \e) > 0$ such that  
\begin{equation}\label{scopo dell articolo 0}
\| u(t)\|_{H^s} \leq C(s, \e)(1 +  | t - t_0|^\e) \| u_0\|_{H^s}\,, \qquad \forall t \in \R\,. 
\end{equation}
\end{theorem}
This theorem will be proved in Section \ref{prova teorema finalissimo} and it will be deduced by the following 
\begin{theorem}[\bf Normal-form theorem]\label{teorema riduzione}
Assume the hypotheses {\bf (H1)}-{\bf (H3)}. For any $K > 0$ there exists a time-dependent symplectic differentiable invertible map $t \mapsto {\cal T}_K(t)$ 
satisfying 
\begin{equation}\label{proprieta cal TK 2}
\sup_{t \in \R} \| {\cal T}_K(t)^{\pm 1}\|_{{\cal B}(H^s)} + \sup_{t \in \R} \| \partial_t {\cal T}_K(t)^{\pm 1}\|_{{\cal B}(H^{s + 1}, H^{s })} < + \infty\,, \qquad \forall s \geq 0
\end{equation}
such that the following holds: the vector field $\ii {\cal V}(t)$ is transformed, by the map ${\cal T}_K$, into the vector field 
\begin{equation}\label{campo finalissimo cal VK}
\ii {\cal V}_K(t) := ({\cal T}_K)_* (\ii {\cal V})(t) = \ii \Big( \lambda_K(t, D) + {\cal W}_K(t) \Big)
\end{equation}
where $\lambda_K(t, D) := {\rm Op}(\lambda_K(t, \xi))$ is a space-diagonal operator with symbol $\lambda_K$ which satisfies
\begin{equation}\label{proprieta lambda K finalissimo}
\lambda_K \in S^M\,, \quad \lambda_K(t, \xi) = \overline{\lambda_K(t, \xi)}\,, \quad \forall (t, \xi) \in \R \times \R
\end{equation}
and 
\begin{equation}\label{proprieta resto finalissimo}
{\cal W}_K(t) = {\rm Op}\Big( w_K(t, x, \xi) \Big)\,, \quad w_K \in S^{- K}
\end{equation}
is $L^2$ self-adjoint. 
\end{theorem}
In the remaining part of the section, we shall explain the main ideas needed to prove Theorems \ref{teo growth of sobolev norms}, \ref{teorema riduzione}. 

\noindent
In order to prove Theorem \ref{teo growth of sobolev norms}, we need to estimate the Sobolev norm $\| u(t) \|_{H^s}$, $s > 0$, for the solutions $u(t)$ of \eqref{cauchy problem main theorem}. Choosing the integer $K \simeq s$ in Theorem \ref{teorema riduzione}, we transform the PDE $\partial_t u = \ii {\cal V}(t) u$ into the PDE $\partial_t v = \ii {\rm Op}\big( \lambda_s(t, \xi) \big) v + O(|D|^{- s}) v$. The Hamiltonian structure guarantees that the symbol $\lambda_s(t, \xi)$ is real. Writing the Duhamel formula for the latter equation, one easily gets that $\| v(t) \|_{H^s} \lesssim_s \| v(t_0) \|_{H^s} + |t - t_0| \| v(t_0) \|_{L^2}$ which implies that the same estimate holds for $u(t)$, i.e. $\| u(t) \|_{H^s} \lesssim_s \| u(t_0) \|_{H^s} + |t - t_0| \| u(t_0) \|_{L^2}$ (in this paper, we use the standard notation $A \lesssim_s B$ if and only if $A \leq C(s) B$ for some constant $C(s)> 0$). Using that $\| u(t) \|_{L^2} = \| u(t_0)\|_{L^2}$ (since ${\cal V}(t)$ is self-adjoint), applying the classical interpolation Theorem \ref{interpolazione sobolev}, one obtains the growth $\sim |t - t_0|^\e$ of the Sobolev norm $\| u(t) \|_{H^s}$ for any $s, \e > 0$.

\noindent
The proof of Theorem \ref{teorema riduzione} is based on a {\it normal form} procedure, which transforms the vector field $\ii {\cal V}(t)$ into another one which is an arbitrarily regularizing perturbation of a {\it space diagonal} vector field. Such a procedure is developed in Section \ref{sezione regolarizzazione cal V (t)} and it is based on symbolic calculus and Egorov type Theorems (see Theorems \ref{Teorema egorov generale}-\ref{teorema egorov campo minore di uno}). We describe below our method in more detail. 

\medskip

\begin{enumerate}
\item{\it Reduction of the highest order.}
Our first aim is to transform the vector field $\ii{\cal V}(t) = \ii \big(V(t, x) |D|^M + {\cal W}(t) \big)$ into another vector field $\ii {\cal V}_1(t)$ whose highest order is $x$-independent, i.e. ${\cal V}_1(t) = \lambda(t) |D|^M + {\cal W}_1(t)$ with ${\cal W}_1 \in OPS^{M - \frak e}$. This is done in Section \ref{riduzione ordine principale M > 1}. In order to achieve this purpose, we transform the vector field $\ii {\cal V}(t)$ by means of the time $1$-flow map of the transport equation 
$$
\partial_\tau u = b_\alpha(\tau; t, x) \partial_x u + \frac{(\partial_x b_\alpha)}{2} u\,, \quad b_\alpha (\tau; t, x):= - \frac{\alpha(t, x)}{1 + \tau \alpha_x(t, x)} \,, \qquad \tau \in [0, 1]
$$
where $\alpha(t, x)$ is a function in ${\cal C}^\infty_b(\R \times \T, \R)$ (to be determined) satisfying $\inf_{(t, x) \in \R \times \T} \Big(1 + (\partial_x \alpha)(t, x) \Big) > 0$. This condition guarantees that $\T \to \T$, $x \mapsto x + \alpha(t, x)$ is a diffeomorphism of the torus with inverse given by $\T \to \T$, $y \mapsto y + \widetilde \alpha(t, y)$ and $\widetilde \alpha \in {\cal C}^\infty_b(\R \times \T, \R)$ satisfying $\inf_{(t, y) \in \R \times \T} \Big( 1 + (\partial_y \widetilde \alpha)(t, y)\Big) > 0$ (see Lemma \ref{diffeo del toro lemma astratto}). The transformed vector field $\ii {\cal V}_1(t)$, ${\cal V}_1(t) = {\rm Op}\Big( v_1(t, x, \xi) \Big)$ is analyzed by using Theorems \ref{Teorema egorov generale}, \ref{parte tempo egorov} and its final expansion is provided in Lemma \ref{espansione simbolo principale egorov}. It turns out that the principal part of the operator ${\cal V}_1(t)$ is given by 
$$
\Big[V(t, y) \big( 1 + \widetilde \alpha_y(t, y) \big)^M \Big]_{y = x + \alpha(t, x)}|D|^M\,.
$$ 
The function $\widetilde \alpha$ is choosen in such a way that $V(t, y) \big( 1 + \widetilde \alpha_y(t, y) \big)^M = \lambda(t)$ where $\lambda \in {\cal C}^\infty_b(\R, \R)$ is independent of $x$ (see \eqref{equazione omologica grado alto}--\eqref{definizione widetilde alpha}). The hypothesis {\bf (H2)} on $V(t, x)$, i.e. $\inf_{(t, x)\in \R \times \T} V(t, x) > 0$ ensures that $\inf_{t \in \R} \lambda(t) > 0$ and $\inf_{(t, y) \in \R \times \T} \Big(1 + (\partial_y \widetilde \alpha)(t, y) \Big) > 0$ and hence also $\inf_{(t, x) \in \R \times \T} \Big(1 + (\partial_x  \alpha)(t, x) \Big) > 0$, by Lemma \ref{diffeo del toro lemma astratto}. 
\item{\it Reduction of the lower order terms.} After the first reduction described above, we deal with a vector field $\ii {\cal V}_1(t)$ where ${\cal V}_1(t) = \lambda(t) |D|^M + {\cal W}_1(t)$ and ${\cal W}_1 \in OPS^{M - \bar{\frak e}}$ for some constant $\bar{\frak e} > 0$. The next step is to transform such a vector field into another one of the form $\ii \Big( \lambda(t) |D| + \mu_N(t, D) + {\cal W}_N(t) \Big)$ where $\mu_N(t, D)$ is a time-dependent Fourier multiplier of order $M- \bar{\frak e}$ and ${\cal W}_N \in OPS^{M - N \bar{\frak e}}$ for any integer $N > 0$. This is proved by means of an iterative procedure developed in Section \ref{Reduction of the lower order terms}, see Proposition \ref{descent method M geq 1}. At the $n$-th step of such a procedure, we deal with a vector field $\ii {\cal V}_n(t)$, ${\cal V}_n(t) = \lambda(t) |D|^M + \mu_n(t, D) + {\cal W}_n(t)$, $\mu_n \in S^{M - \bar{\frak e}}$, ${\cal W}_n \in OPS^{M - n \bar{\frak e}}$. We transform such a vector field by means of the time-$1$ flow map of the PDE 
$$
\partial_\tau u = \ii {\cal G}_n(t) [u] \quad \text{where} \quad {\cal G}_n(t) = {\cal G}_n(t)^*\,, \quad {\cal G}_n \in OPS^{1 - n \bar{\frak e}}\,. 
$$
Using Theorem \ref{teorema egorov campo minore di uno}, the transformed vector field $\ii {\cal V}_{n + 1}(t)$, ${\cal V}_{n + 1}(t) = {\rm Op}\Big( v_{n + 1}(t, x, \xi)\Big)$ has the symbol expansion 
$$
v_{n + 1}(t, x, \xi) = \lambda(t) |\xi|^M  + \mu_n(t, \xi)+ w_n(t, x, \xi) - M \lambda(t) |\xi|^{M - 2} \xi   \partial_x g_n(t, x, \xi) + O(|\xi|^{M - (n + 1) \bar{\frak e}})$$
One then finds $g_n(t, x, \xi)$ so that $g_n = g_n^*$ ($g_n^*$ is the symbol of the adjoint operator ) and which solves the equation 
$$
w_n(t, x, \xi) - M \lambda(t) |\xi|^{M - 2} \xi   \partial_x g_n(t, x, \xi) = \langle w_n \rangle_x(t, \xi) + O(|\xi|^{M - (n + 1) \bar{\frak e}})
$$
where $\langle w_n\rangle_x(t, \xi) := \frac{1}{2 \pi}\int_\T w_n(t, x, \xi)\, d x$ (see Lemma \ref{equazione omologica ordini bassi}). This implies that the transformed symbol has the form $v_{n + 1}(t, x, \xi) = \lambda(t) |\xi|^M + \mu_{n + 1}(t, \xi) + O(|\xi|^{M - (n + 1) \bar{\frak e}})$ with $\mu_{n + 1} = \mu_n + \langle w_n \rangle_x$.  
\end{enumerate}

\bigskip

\noindent
The paper is organized as follows: in Section \ref{sezione technical tools} we provide some technical tools which are needed for the proof of Theorem \ref{teorema riduzione}. In Section \ref{sezione regolarizzazione cal V (t)} we develop the regularization procedure of the vector field that we use in Section \ref{prova teorema riduzione conclusa} to deduce Theorem  \ref{teorema riduzione}. Finally, in Section \ref{prova teorema finalissimo} we prove Theorem \ref{teo growth of sobolev norms}. 

\medskip

\noindent
\noindent
{\it Acknowledgements}. The author warmly thanks Giuseppe Genovese, Emanuele Haus, Thomas Kappeler, Felice Iandoli and Alberto Maspero for many useful discussions and comments.  

\section{Pseudo-differential operators}\label{sezione technical tools}
In this section, we recall some well-known definitions and results concerning pseudo differential operators on the torus $\T$. We always consider time dependent symbols $a(t, x, \xi)$ depending in a ${\cal C}^\infty$ way on the whole variables, see Definitions \ref{classe simboli}, \ref{def:Ps2}. Actually the time $t$ is only a parameter, hence all the classical results apply without any modification (we refer for instance to \cite{SV}, \cite{Taylor}). 

\noindent
For the symbol class $S^m$ given in the definition \ref{classe simboli} and the operator class $OPS^m$ given in the definition \ref{def:Ps2}, the following standard inclusions hold: 
\begin{equation}\label{inclusioni Sm OPSm}
S^m \subseteq S^{m'}\,, \quad OPS^m \subseteq OPS^{m'}\,, \quad \forall m \leq m'\,. 
\end{equation}
We define the class of smoothing symbol and smoothing operators $S^{- \infty} := \cap_{m \in \R} S^m$, $OPS^{- \infty} := \cap_{m \in \R} OPS^m$
\begin{theorem}[Calderon-Vallancourt]\label{conitnuita pseudo}
Let $m \in \R$ and $A = a( t, x, D) \in OPS^m$. Then for any $s \in \R$, for any $\alpha \in \N$ the operator $\partial_t^\alpha A(t) \in {\cal B}(H^{s + m}(\T), H^s(\T))$ with $\sup_{t \in \R} \| \partial_t^\alpha A(t)\|_{{\cal B}(H^{s + m}, H^s)} < + \infty$. 
\end{theorem}
\begin{definition}[\bf Asymptotic expansion]\label{definizione espansione asintotica}
Let $(m_k)_{k \in \N}$ be a strictly decreasing sequence of real numbers converging to $- \infty$ and $a_k \in S^{m_k}$ for any $k \in \N$. We say that $a \in S^{m_0}$ has the asymptotic expansion $\sum_{k \geq 0} a_k$, i.e. 
$$
a \sim \sum_{k \geq 0} a_k
$$
if for any $N \in \N$
$$
a - \sum_{k = 0}^{N} a_k \in S^{m_{N + 1}}\,. 
$$
\end{definition}
Given a symbol $a \in S^m$, we denote by $\widehat a$, the Fourier transform with respect to the variable $x$, i.e. 
\begin{equation}\label{trasformata in x simbolo}
\widehat a(t, \eta, \xi) := \frac{1}{2 \pi} \int_\T a(t, x, \xi) e^{- \ii \eta x}\, d x \,, \quad (t, \eta, \xi) \in \R \times \Z \times \R\,. 
\end{equation}
\begin{theorem}[{\bf Composition}]\label{teorema composizione pseudo}
Let $m, m' \in \R$ and $ A = a( t, x, D) \in OPS^{m} $, $ B = b(t, x, D) \in  OPS^{m'} $. Then the composition operator 
$ A B := A \circ B = \sigma_{AB} ( t, x, D) $ is a pseudo-differential operator in $OPS^{m + m'}$ with symbol
\be\label{lemma composition}
\sigma_{AB}( t, x, \xi) = \sum_{\eta \in \Z} a(t, x, \xi + \eta ) \widehat b(t, \eta, \xi) e^{\ii \eta x  }\,.
\ee
The symbol $\sigma_{AB} $ has the following asymptotic expansion
\be\label{composition pseudo}
\sigma_{AB}(t, x, \xi) \sim {\mathop \sum}_{\b \geq 0} \frac{1}{\ii^\b \b !} \partial_\xi^\b a(t, x, \xi)  \partial_x^\b b(t, x, \xi) \, , 
\ee
that is, $ \forall N \geq 1  $,  
\be\label{expansion symbol}
\s_{AB} (t, x, \xi) = \sum_{\b =0}^{N-1} \frac{1}{  \b ! \ii^\b }  \pa_\xi^\b a (t, x, \xi) \, \pa_x^\b b (t, x, \xi) + r_N (t, x, \xi)
\qquad {\rm where} \qquad r_N := r_{N, AB} \in S^{m + m' - N }  \, .
\ee
The remainder $ r_N $ has the explicit formula 
\be\label{rNTaylor}
r_N (t, x, \xi) := \frac{1}{(N-1)! \, \ii^N} \int_0^1 (1- \tau )^{N-1}  
\sum_{\eta \in \Z} (\pa_\xi^N a)(t, x, \xi + \tau \eta ) \widehat {\pa_x^N b} (t, \eta , \xi) e^{\ii \eta x } \, d \tau  \, .
\ee

\end{theorem}
\begin{corollary}\label{corollario commutator}
Let $m, m' \in \R$ and let $A = {\rm Op}(a)$, $B = {\rm Op}(b)$. Then the commutator $[A, B] = {\rm Op}(a \star b)$, with $a \star b \in S^{m + m' - 1}$ having the following expansion: 
$$
a \star b = - \ii \{a, b \} + \mathtt r_2(a, b)\,, \quad \{ a, b \} := \partial_\xi a \partial_x b - \partial_x a \partial_\xi b \in S^{m + m'-1} \,, \quad \mathtt r_2(a, b) \in S^{m + m' - 2}\,.
$$
\end{corollary}
\begin{theorem}[{\bf Adjoint of a pseudo-differential operator}]\label{adjoint}
If  $ A(t) =  a(t, x, D)  \in OPS^m $ is a pseudo-differential operator with symbol $ a \in S^m $, then its $L^2$-adjoint is the
pseudo-differential operator $A^* = {\rm Op}(a^*) \in OPS^m$ defined by
\begin{equation}\label{simbolo aggiunto senza tempo}
A^* = {\rm Op}(a^*) \qquad {\rm with \ symbol}  
\qquad a^*(t, x, \xi) := \overline{{\mathop \sum}_{\eta \in \Z} \widehat a(t, \eta, \xi - \eta) e^{\ii \eta x }}\,.
\end{equation}
The symbol $a^* \in S^m$ admits the asymptotic expansion 
\begin{equation}\label{espansione asintotica aggiunto}
a^*(t, x, \xi) \sim \sum_{\alpha \in \N} \frac{(- \ii)^{\alpha}}{\alpha !}  \overline{\partial_x^\alpha \partial_\xi^\alpha a(t, x, \xi )}
\end{equation}
meaning that for any integer $N \geq 1$, 
$$
a^*(t, x, \xi) = \sum_{\alpha = 0}^{N - 1} \frac{(- \ii)^{\alpha}}{\alpha !}  \overline{\partial_x^\alpha \partial_\xi^\alpha a(t, x, \xi )} + r_N^{*}(t, x, \xi)\qquad {\rm where} \qquad r_N^* := r_{N, A}^* \in S^{m  - N }  \, .
$$
The remainder $r_N^*$ has the explicit formula 
\begin{equation}\label{resto pseudo operatore aggiunto}
r_N^*(t, x, \xi) := \frac{(- 1)^N}{(N - 1)!} \int_0^1 (1 - \tau)^{N - 1} \overline{\sum_{\eta \in \Z} \widehat{\partial_\xi^N a}(t, \eta, \xi - t \eta) \eta^N e^{\ii \eta x}} \, d \tau\,. 
\end{equation}
\end{theorem}
Note that if $a \in S^m$ is a symbol independent of $x$ (Fourier multiplier) then 
\begin{equation}\label{aggiunto Fourier multiplier}
a^*(t, \xi) = \overline{a(t, \xi)}\,, \quad \forall (t, \xi ) \in \R \times \R\,. 
\end{equation}
We now prove some useful lemmas which we apply in Section \ref{sezione regolarizzazione cal V (t)}. 
\begin{lemma}\label{simbolo autoaggiunto fourier multiplier}
Let $A = {\rm Op}(a) \in OPS^m$ be self-adjoint, i.e. $a(t, x, \xi) = a^*(t, x, \xi)$ and let $\vphi(t, \xi)$ be a real Fourier multiplier of order $m'$. We define the symbol $b(t, x, \xi) := \vphi(t, \xi) a(t, x, \xi) \in S^{m + m'}$. Then 
$$
b^*(t, x, \xi) - b(t, x, \xi) \in S^{m + m' - 1}.
$$
\end{lemma}
\begin{proof}
One has that 
$$
{\rm Op}(b^*) = {\rm Op}(b)^* = {\rm Op}(\vphi)^* \circ {\rm Op}(a)^*\,. 
$$
Since ${\rm Op}(a)$ is self-adjoint and since $\lambda$ is real, one has that 
$$
{\rm Op}(b^*)  =  {\rm Op}(\vphi) \circ {\rm Op}(a)\,. 
$$
By applying Theorem \ref{teorema composizione pseudo}, one gets that 
$$
{\rm Op}(b^*) = {\rm Op}\Big(\vphi(t, \xi) a(t, x, \xi) + r(t, x, \xi) \Big)\,, \qquad r \in S^{m + m' - 1}
$$
and then the lemma is proved. 
\end{proof}
 We define the operator $\partial_x^{- 1}$ by setting 
 \begin{equation}\label{definizione partial x - 1}
 \partial_x^{- 1}[1] := 0\,, \quad \partial_x^{- 1}[e^{\ii x k  }] := \frac{e^{\ii x k }}{\ii k}\,, \quad \forall k \in  \Z \setminus \{ 0 \}\,. 
 \end{equation}
 Furthermore, given a symbol $a \in S^m$, we define the averaged symbol $\langle a \rangle_x$ by 
 \begin{equation}\label{simbolo mediato}
 \langle a \rangle_x(t, \xi) := \frac{1}{2 \pi} \int_\T a(t, x, \xi)\, d x \,, \quad \forall (t, \xi) \in \R \times \R\,. 
 \end{equation}
The following elementary property holds: 
\begin{equation}\label{proprieta simbolo mediato e partial x - 1}
a \in S^m \Longrightarrow \partial_x^{- 1} a \,,\, \langle a \rangle_x \in S^m\,. 
\end{equation}
We now prove the following 
\begin{lemma}\label{lemma aggiunto media partial x - 1}
Let $a \in S^m$. Then the following holds: 

\noindent
$(i)$ $\langle a^* \rangle_x = (\langle a \rangle_x)^* = \langle \overline a \rangle_x$,

\noindent
$(ii)$ $\partial_x^{- 1} (a^*) = (\partial_x^{- 1} a)^*$.
\end{lemma}
\begin{proof} 
{\sc Proof of $(i)$.} By Theorem \ref{adjoint} and since by the definition \eqref{simbolo mediato} the symbol $\langle a \rangle_x$ is $x$-independent, one has that 
\begin{equation}\label{a a* x nel lemma 0}
(\langle a \rangle_x)^*(t, \xi) = {\langle \overline a \rangle_x(t, \xi)}\,.
\end{equation}
Moreover by \eqref{simbolo aggiunto senza tempo}, \eqref{simbolo mediato} one gets 
\begin{align}
\langle a^* \rangle_x(t, \xi) & = \frac{1}{2 \pi} \int_\T \Big( \overline{{\mathop \sum}_{\eta \in \Z} \widehat a(t, \eta, \xi - \eta) e^{\ii \eta x }} \Big)\, d x= \overline{\widehat a(t, 0, \xi)} \nonumber\\
& \stackrel{\eqref{trasformata in x simbolo}}{=} \overline{\frac{1}{2\pi} \int_\T a(t, x, \xi)\, d x }\stackrel{\eqref{simbolo mediato}}{=} \overline{\langle a \rangle_x(t, \xi)} \stackrel{\eqref{a a* x nel lemma 0}}{=} (\langle a \rangle_x)^*(t, \xi)\, \nonumber
\end{align}
hence the claimed statement follows.  

\medskip

\noindent
{\sc Proof of $(ii)$.} By \eqref{definizione partial x - 1}, one has that 
$$
\widehat{\partial_x^{- 1} a}(t, 0, \xi) = 0\,, \quad \widehat{\partial_x^{- 1} a}(t, \eta, \xi) = \frac{\widehat a(t, \eta, \xi)}{\ii \eta}, \quad \eta \in \Z \setminus \{ 0 \},
$$
hence by formula \eqref{simbolo aggiunto senza tempo}
\begin{align}\label{a a* partial x - 1 nel lemma}
(\partial_x^{- 1} a)^*(t, x, \xi)  & = \overline{{\mathop \sum}_{\eta \in \Z \setminus \{ 0 \}} \frac{\widehat a(t, \eta, \xi - \eta)}{\ii \eta} e^{\ii \eta x }} = \overline{{\mathop \sum}_{\eta \in \Z } \widehat a(t, \eta, \xi - \eta) \partial_x^{- 1}(e^{\ii \eta x })} \nonumber\\
& = \partial_x^{- 1} \Big( \overline{{\mathop \sum}_{\eta \in \Z } \widehat a(t, \eta, \xi - \eta) e^{\ii \eta x }} \Big) = \partial_x^{- 1} (a^*)(t, x, \xi)
\end{align}
which proves item $(ii)$. 
\end{proof}
For any $\alpha \in \R$, the operator $|D|^\alpha$, acting on $2 \pi$-periodic functions $u(x) = \sum_{\xi \in \Z} \widehat u(\xi) e^{\ii x \xi}$ is defined by  
$$
|D|^\alpha u(x) := \sum_{\xi \in \Z \setminus \{ 0 \}} |\xi|^\alpha \widehat u(\xi) e^{\ii x \xi}\,. 
$$
We shall identify the operator $|D|^\alpha$ with the operator associated to a Fourier multiplier $|\xi|^\alpha\chi(\xi)$ in $S^\alpha$ where $\chi \in {\cal C}^\infty(\R, \R)$ is an even cut-off function satisfying 
\begin{equation}\label{definizione cut off D alpha}
\chi(\xi) := \begin{cases}
1 & \quad \text{if} \quad |\xi| \geq 1 \\
0 & \quad \text{if} \quad |\xi| \leq \frac12\,.\\
\end{cases}
\end{equation}
Then, for any $\alpha \in \R$,
\begin{equation}\label{definizione D alpha}
|D|^\alpha \equiv {\rm Op}\big( |\xi|^\alpha \chi(\xi) \big)
\end{equation}
since  the action of the two operators on $2\pi$-periodic functions $u \in L^2(\T)$ coincides. 

\medskip

\noindent
We conclude this section by stating an interpolation theorem, which is an immediate consequence of the classical Riesz-Thorin interpolation theorem in Sobolev spaces. 
\begin{theorem}\label{interpolazione sobolev}
Let $0 \leq s_0 < s_1$ and let $A \in {\cal B}(H^{s_0}) \cap {\cal B}(H^{s_1})$. Then for any $s_0 \leq s \leq s_1$ the operator $A \in {\cal B}(H^s)$ and 
$$
\| A \|_{{\cal B}(H^s)} \leq \| A\|_{{\cal B}(H^{s_0})}^\lambda \| A\|_{{\cal B}(H^{s_1})}^{1 - \lambda}\,, \quad \lambda := \frac{s_1 - s}{s_1 - s_0}\,. 
$$
\end{theorem}

\subsection{Well posedness of some linear PDEs}
In this section we study the properties of the flow of some linear pseudo-PDEs. We start with the following lemma. 
\begin{lemma}\label{flusso + derivate flusso}
Let ${\cal A}(\tau; t) := {\rm Op}\Big( a(\tau;t,  x, \xi) \Big)$, $\tau \in [0, 1]$ be a smooth $\tau$-dependent family of pseudo differential operators in $OPS^1$.  Assume that ${\cal A}(\tau; t) + {\cal A}(\tau; t)^* \in OPS^{0}$. Then the following holds. 

\noindent
$(i)$ Let $s \geq 0$, $u_0 \in H^s(\T)$, $\tau_0 \in [0, 1]$. Then there exists a unique solution $u \in {\cal C}_b^0\Big([0, 1], H^s(\T) \Big)$ of the Cauchy problem 
\begin{equation}\label{rafael}
\begin{cases}
\partial_\tau u = {\cal A}(\tau; t)[ u ] \\
u(t_0, x) = u_0(x)
\end{cases}
\end{equation}
satisfying the estimate 
$$
\| u\|_{{\cal C}^0_b([0, 1], H^s)} \lesssim_s \| u_0\|_{H^s} \,.
$$
As a consequence, for any $\tau \in [0, 1]$, the flow map $\Phi(\tau_0, \tau; t)$, which maps the initial datum $u(\tau_0 ) = u_0$ into the solution $u(\tau)$ of \eqref{rafael} at the time $\tau$, is in ${\cal B}(H^s)$ with $\sup_{\begin{subarray}{c}
\tau_0, \tau \in [0, 1] \\
t \in \R
\end{subarray}} \| \Phi(\tau_0, \tau; t)\|_{{\cal B}(H^s)} < + \infty$ for any $s \geq 0$. Moreover, the operator $\Phi(\tau_0, \tau; t)$ is invertible with inverse $\Phi(\tau_0, \tau; t)^{- 1} = \Phi(\tau, \tau_0; t)$. 

\medskip

\noindent
$(ii)$ For any $\tau_0, \tau \in [0, 1]$,  the flow map $ t \mapsto \Phi(\tau_0, \tau; t)$
is differentiable and 
\begin{equation}\label{rafael - 10}
 \sup_{\begin{subarray}{c}
\tau_0, \tau \in [0, 1] \\
t \in \R
\end{subarray}} \| \partial_t^k \Phi(\tau_0, \tau; t)\|_{{\cal B}(H^{s + k}, H^{s})} < + \infty\,, \quad \forall k \in \N, \quad s \geq 0\,.  
\end{equation}
\end{lemma}
\begin{proof}
{\sc Proof of $(i)$.} The proof of item $(i)$ is classical. We refer for instance to \cite{Taylor}, Section 0.8. 

\medskip

\noindent
{\sc Proof of $(ii)$.} For any $\tau_0 \in [0, 1]$, the flow map $\Phi(\tau_0, \tau; t)$ solves   
\begin{equation}\label{rafael 0}
\begin{cases}
\partial_\tau \Phi(\tau_0, \tau ; t) = {\cal A}(\tau; t) \Phi(\tau_0, \tau; t) \\
\Phi(\tau_0, \tau_0 ; t) = {\rm Id}\,. 
\end{cases}
\end{equation}
By differentiating \eqref{rafael 0} with respect to $t$, one gets that $\partial_t \Phi(\tau_0, \tau; t)$ solves 
$$
\begin{cases}
\partial_\tau \Big( \partial_t\Phi(\tau_0, \tau; t) \Big) = {\cal A}(\tau; t) \Big(\partial_t\Phi(\tau_0, \tau ; t) \Big) + \Big(\partial_t {\cal A}(\tau; t) \Big) \Phi(\tau_0, \tau ; t) \\
\partial_t\Phi(\tau_0, \tau_0 ; t) = 0\,.
\end{cases}
$$
By Duhamel principle, we then get 
$$
\partial_t\Phi(\tau_0, \tau; t) = \int_{\tau_0}^\tau \Phi(\tau_0, \tau; t) \Phi(\zeta, \tau_0; t) \partial_t {\cal A}(\zeta; t) \Phi(\tau_0, \zeta; t)\, d \zeta\,.
$$
By item $(i)$ and by Theorem \ref{conitnuita pseudo} (using that $\partial_t {\cal A} \in OPS^1$) one gets that $\partial_t \Phi(\tau_0, \tau; t ) \in {\cal B}(H^{s + 1}, H^{s })$ with estimates which are uniform with respect to $\tau_0, \tau \in [0, 1]$ and $t \in \R$. Hence \eqref{rafael - 10} has been proved for $k = 1$. Iterating the above argument, one can prove the estimate \eqref{rafael - 10} for any positive integer $k$. 
\end{proof}
In the next lemma we prove the global well-posedness for a class of Schr\"odinger type equations. Let $\vphi(t, \xi) \in OPS^m$ be a real Fourier multiplier, i.e. 
\begin{equation}\label{cauchy astratto fourier multiplier}
\vphi \in S^m\,, \quad \vphi(t, \xi) = \overline{\vphi(t, \xi)}\,, \qquad \forall (t, \xi) \in \R \times \R\,. 
\end{equation}
Moreover, let us consider a time dependent linear operator $t \mapsto {\cal R}(t)$ satisfying 
\begin{equation}\label{resto astratto cauchy}
{\cal R} \in {\cal C}^0_b(\R, {\cal B}(H^s))\,,  \qquad \forall s \geq 0\,. 
\end{equation}
The following lemma holds: 
\begin{lemma}
Let $s \geq 0$, $u_0 \in H^s(\T)$, $t_0 \in \R$. Then there exists a unique global solution $u \in {\cal C}^0(\R, H^s)$ of the Cauchy problem 
\begin{equation}\label{cauchy schrodinger astratto}
\begin{cases}
\partial_t u + \ii \vphi(t, D) u + {\cal R}(t)[u] = 0 \\
u(t_0, x) = u_0(x)\,. 
\end{cases}
\end{equation}
\end{lemma}
\begin{proof}
The local existence follows by a fixed point argument applied to the map 
$$
{\cal F}(u) := {\rm exp}\Big(-\ii \Phi (t, D)  \Big) [u_0] + \int_{t_0}^t {\rm exp}\Big(-\ii \Big(\Phi (t, D) - \Phi(\tau, D)\Big) \Big) {\cal R}(\tau)[u(\tau, \cdot)]\, d \tau
$$
where 
$$
\Phi(t, D) := {\rm Op}(\Phi(t, \xi)) \,, \qquad \Phi(t, \xi) := \int_{t_0}^t\vphi(\zeta, \xi)\, d \zeta\,. 
$$
Since $\vphi(t, \xi)$ is real, then also $\Phi(t, \xi)$ is real, implying that the propagator $ {\rm exp}\Big(-\ii \Phi (t, D)  \Big)$ is unitary on Sobolev spaces. 
Choosing 
$$
R := 2 \| u_0\|_{H^s}\,, \qquad T := \frac{1}{2 \| {\cal R}\|_{{\cal C}^0_b(\R, {\cal B}(H^s))}}
$$ 
and defining
$$
{\cal B}_{R, T}(s) := \Big\{ u \in {\cal C}^0_b([t_0 - T, t_0 + T], H^s) : \| u \|_{{\cal C}^0_b([t_0 - T, t_0 + T], H^s)} \leq R \Big\}
$$
one can prove that 
$$
{\cal F} : {\cal B}_{R, T}(s) \to {\cal B}_{R, T}(s)
$$
is a contraction. The global well posedness follows from the fact that the solution is bounded on any bounded interval and then it can be extended to the whole real line. 
\end{proof}

\subsection{Some Egorov-type theorems}\label{sezione astratta Egorov}
In this section we collect some abstract egorov type theorems, namely we study how a pseudo differential operator transforms under the action of the flow of a first order hyperbolic PDE. 
 Let $\alpha : \R \times \T \to \R$ be a ${\cal C}^\infty$ function with all the derivatives bounded, satisfying  
 \begin{equation}\label{ansatz alpha egorov}
\alpha \in {\cal C}^\infty_b(\R \times \T, \R) \,,\quad \inf_{(t, x) \in \R \times \T} \big( 1 + \alpha_x(t, x) \big) > 0\,. 
 \end{equation} 
 We then consider the non-autonomous transport equation
\begin{equation}\label{trasporto per Egorov}
\partial_\tau u = {\cal A}(\tau; t, x, D) u\,, \qquad {\cal A}(\tau; t) = {\cal A}(\tau; t,  x, D) := b_{\alpha}(\tau; t, x) \partial_x + \frac{(\partial_x b_{\alpha})(\tau; t, x)}{2}\,, \quad 
\end{equation}
\begin{equation}\label{definizione b alpha per egorov}
b_\alpha (\tau; t, x):= - \frac{\alpha(t, x)}{1 + \tau \alpha_x(t, x)} \,, \qquad \tau \in [0, 1]\,. 
\end{equation}
Note that the condition \eqref{ansatz alpha egorov} implies that 
$$
\inf_{\begin{subarray}{c}
\tau \in [0, 1] \\
(t, x) \in \R \times \T
\end{subarray}} \big( 1+ \tau \alpha_x(t, x) \big) > 0\,,
$$
hence the function $b \in {\cal C}^\infty_b([0, 1] \times \R \times \T)$. Then ${\cal A}(\tau; \cdot) \in OPS^1$, $\tau \in [0, 1]$ is a smooth family of pseudo-differential operators and it is straightforward to see that ${\cal A}(\tau; t) + {\cal A}(\tau; t)^* = 0$. Therefore, the hypotheses of Lemma \ref{flusso + derivate flusso} are verified, implying that, for any $\tau \in [0, 1]$, the flow $\Phi(\tau; t) \equiv \Phi(0, \tau; t)$, $\tau \in [0, 1]$ of the equation \eqref{trasporto per Egorov}, i.e. 
\begin{equation}\label{equazione flusso operatoriale}
\begin{cases}
\partial_\tau \Phi(\tau; t) = {\cal A}(\tau; t) \Phi(\tau; t)  \\
\Phi (0; t) = {\rm Id}
\end{cases}
\end{equation}
is a well defined map and satisfies all the properties stated in the items $(i)$, $(ii)$ of Lemma \ref{flusso + derivate flusso}. 
Furthermore, ${\cal A}(\tau; t)$ is a Hamiltonian vector field. Indeed 
\begin{equation}\label{proprieta hamiltoniana flusso trasporto}
\begin{aligned}
& {\cal A}(\tau; t) = \ii \widetilde {\cal A}(\tau; t)\,, \quad \widetilde {\cal A}(\tau; t) := - \ii \Big( b_{\alpha}(\tau; t, x) \partial_x + \frac{(\partial_x b_{\alpha})(\tau; t, x)}{2} \Big) \\
& \quad \text{and} \quad \widetilde {\cal A}(\tau; t)= \widetilde {\cal A}(\tau; t)^*
\end{aligned}
\end{equation}
implying that the map $\Phi(\tau; t)$ is symplectic. We then have the following 
\begin{lemma}\label{proprieta flusso trasporto lemma astratto}
The flow $\Phi(\tau; t)$ given by \eqref{equazione flusso operatoriale} is a symplectic, invertible map satisfying 
$$
\sup_{\begin{subarray}{c}
\tau \in [0, 1] \\
t \in \R
\end{subarray}}\| \partial_t^k \Phi(\tau; t)^{\pm 1}\|_{{\cal B}(H^{s + k}, H^{s})} < + \infty\,, \quad \forall k \in \N, \quad s \geq 0\,.  
$$ 
\end{lemma}
In order to state Theorem \ref{Teorema egorov generale} of this section, we need some preliminary results. 
 \begin{lemma}\label{diffeo del toro lemma astratto}
 Let $\alpha \in {\cal C}^\infty_b(\R \times \T, \R)$ satisfy the condition \eqref{ansatz alpha egorov}. Then for any $t \in \R$, the map 
 $$
 \vphi_t : \T \to \T\,, \quad x \mapsto x + \alpha(t, x)
 $$
 is a diffeomorphism of the torus whose inverse has the form 
 \begin{equation}\label{forma vphi t - 1}
 \vphi_t^{- 1} : \T \to \T , \quad y \mapsto y + \widetilde \alpha(t, y)\,,
 \end{equation}
 with $\widetilde \alpha : \R \times \T \to \R$ satisfying 
 \begin{equation}\label{proprieta alpha tilde diffeo toro}
\widetilde \alpha \in {\cal C}^\infty_b(\R \times \T, \R)\,, \quad \inf_{(t, y) \in \R \times \T} \big(1 + \widetilde \alpha_y(t, y) \big) > 0\,. 
 \end{equation}
 Furthermore, the following identities hold: 
 \begin{equation}\label{identita 1 + alpha alpha tilde partial}
 \begin{aligned}
1 + \alpha_x(t, x) = \frac{1}{1 + \widetilde \alpha_y\big(t, x + \alpha(t, x)\big)}\,, \quad 1 + \widetilde \alpha_y(t, y) = \frac{1}{1 + \alpha_x\big(t, y + \widetilde \alpha(t, y)\big)}
 \end{aligned}
 \end{equation}
 \end{lemma}
 \begin{proof}
 The condition \eqref{ansatz alpha egorov} and the inverse function theorem imply that for any $t \in \R$, the map $\vphi_t : \R \to \R$ is a ${\cal C}^\infty$ diffeomorphism with a ${\cal C}^\infty$ inverse given by $\vphi_t^{- 1} : \R \to \R$. Since $\alpha$ is $2 \pi$-periodic in $x$ one verifies easily that $\vphi_t(x + 2  \pi) = \vphi_t(x) + 2  \pi$, implying that $\vphi_t : \T \to \T$ is a diffeomorphism of the torus. We now verify that $\vphi_t^{- 1}$ has the form \eqref{forma vphi t - 1}. In order to see this, it is enough to show that $\widetilde \alpha(t, y) := \vphi_t^{- 1}(y) - y$ is $2 \pi$-periodic in $y$. Let $y = \vphi_t(x)$. Applying $\vphi_t^{- 1}$ to both sides of the equality $\vphi_t(x + 2 \pi) = \vphi_t(x) + 2  \pi$, one gets that $x + 2 \pi = \vphi_t^{- 1}(y + 2  \pi)$, i.e. $\vphi_t^{- 1}(y) + 2  \pi = \vphi_t^{- 1}(y + 2  \pi)$. This implies that 
 $$
 \widetilde \alpha(t, y + 2 \pi) =  \vphi_t^{- 1}(y + 2  \pi) - y - 2 \pi = \vphi_t^{- 1}(y) + 2  \pi - y - 2 \pi = \widetilde \alpha(t, y)
 $$
 and then $\widetilde \alpha$ is $2 \pi$-periodic in $y$. Since 
 $$
 y = x + \alpha(t, x) \iff x = y + \widetilde \alpha(t, y)
 $$
 one has 
 \begin{equation}\label{pachito 0}
 \widetilde \alpha(t, y) + \alpha(t, y + \widetilde \alpha(t, y)) = 0\,, \quad \forall (t, y) \in \R \times \T\,. 
 \end{equation}
 It follows by the standard implicit function theorem that $\widetilde \alpha$ is ${\cal C}^1$ with derivatives
 \begin{equation}\label{formula derivate alpha alpha tilde}
 \partial_y \widetilde \alpha(t, y) = - \frac{\alpha_x(t, y + \widetilde \alpha(t, y))}{1 + \alpha_x(t, y + \widetilde \alpha(t, y))} \,, \qquad \partial_t \widetilde \alpha(t, y) = - \frac{\alpha_t(t, y + \widetilde \alpha(t, y))}{1 + \alpha_x(t, y + \widetilde \alpha(t, y))}\,.
 \end{equation}
 By induction, it can be proved that $\widetilde \alpha$ is ${\cal C}^\infty$ with all the derivatives bounded, namely $\widetilde \alpha \in {\cal C}^\infty_b (\R \times \T, \R)$. The identities \eqref{identita 1 + alpha alpha tilde partial} follow easily by \eqref{formula derivate alpha alpha tilde} and then also \eqref{proprieta alpha tilde diffeo toro} holds. 
The proof of the lemma is then concluded. 
  \end{proof}
  In the next we study the flow of the ODE
\begin{equation}\label{flusso caratteristiche}
\begin{cases}
\dot x(\tau) = - b_\alpha (\tau; t, x(\tau))\\
\dot \xi(\tau) = \partial_x b_\alpha (\tau; t, x(\tau))\xi(\tau)
\end{cases}
\end{equation}
where $b_\alpha(\tau; t, x)$ is defined in \eqref{definizione b alpha per egorov}. 
Given $\tau_0, \tau_1 \in [0, 1]$, we denote by $\gamma^{\tau_0, \tau_1}(t, x, \xi) = (\gamma_1^{\tau_0, \tau_1}(t, x), \gamma_2^{\tau_0, \tau_1}(t, x, \xi))$ the flow of the ODE \eqref{flusso caratteristiche} with initial time $\tau_0$ and final time $\tau_1$. 
We point out that the first equation in \eqref{flusso caratteristiche} is independent of $\xi$, hence the first component of the flow is independent of $\xi$ too. 
We now prove the following lemma concerning the characteristic equation \eqref{flusso caratteristiche}. 
\begin{lemma}\label{espressione esplicita caratteristiche}
For any $\tau_0, \tau \in [0, 1]$, $\gamma_1^{\tau_0, \tau}\in {\cal C}^\infty_b(\R \times \T, \R)$  and $\gamma^{\tau_0, \tau}_2\in S^1$.   Furthermore, for any $\tau_0 \in [0, 1]$ one has that 
$$
\gamma_1^{\tau_0, 0}(t, x) = x + \tau_0 \alpha(t, x) \,, \qquad \gamma_2^{\tau_0, 0}(t, x, \xi) = (1 + \tau_0 \alpha_x(t, x))^{- 1}\xi \,. 
$$
\end{lemma}
\begin{proof}
Given $\tau_0 \in [0 , 1]$, we consider the Cauchy problem 
\begin{equation}\label{Cauchy caratteristiche nel lemma}
\begin{cases}
\dot x(\tau) = - b_\alpha (\tau; t, x(\tau))\,, \qquad x(\tau_0) = x\\
\dot \xi(\tau) = \partial_x b_\alpha (\tau; t, x(\tau))\xi(\tau)\,, \qquad \xi(\tau_0) = \xi\, . 
\end{cases}
\end{equation}
Let $x(\tau) = \gamma_1^{\tau_0, \tau}(t, x)$, $\xi(\tau) = \gamma_2^{\tau_0, \tau}(t, x, \xi)$ be the unique solution of \eqref{Cauchy caratteristiche nel lemma}. 
The second equation can be integrated explicitly, leading to 
\begin{equation}\label{hamsik 0}
\xi(\tau) = \gamma_2^{\tau_0, \tau}(t, x, \xi) = {\rm exp} \Big( \int_{\tau_0}^\tau \partial_x b_\alpha \big(\zeta; \gamma_1^{\tau_0, \zeta}(t, x) \big)\, d \zeta  \Big) \xi\,, \qquad \forall \tau \in [0 , 1]\,. 
\end{equation}
Note that, since $b_\alpha$ is ${\cal C}^\infty$ with respect to all its variables and all its derivatives are bounded, by the smooth dependence of the flow on the initial data $(x, \xi)$ and on the parameter $t$, one has that $\gamma_1^{\tau_0, \tau}$ is ${\cal C}^\infty$ w.r. to $(t,x)$ with all bounded derivatives. Hence by \eqref{hamsik 0} one gets that $\gamma_2^{\tau_0, \tau} \in S^1$. 
By differentiating with respect to the initial datum $x$ the first equation in \eqref{Cauchy caratteristiche nel lemma} one gets that 
$$
\partial_\tau \big( \partial_x \gamma_1^{\tau_0, \tau}(x) \big) = - \partial_x b_\alpha \big(\tau;   \gamma_1^{\tau_0, \tau}(x) \big) \partial_x \gamma_1^{\tau_0, \tau}(x)\,, \quad \partial_x \gamma_1^{\tau_0, \tau_0}(x) = 1 
$$
whose solution is given by 
\begin{equation}\label{hamsik 1}
\partial_x \gamma_1^{\tau_0, \tau}(x) = {\rm exp}\Big( - \int_{\tau_0}^\tau \partial_x b_\alpha\big(\zeta;   \gamma_1^{\tau_0, \zeta}(x) \big) \, d \zeta \Big)\,.
\end{equation}
By formulae \eqref{hamsik 0}, \eqref{hamsik 1}, one then obtains  that 
\begin{equation}\label{hamsik 2}
\gamma_2^{\tau_0, \tau}(x, \xi) = \Big( \partial_x \gamma_1^{\tau_0, \tau}(x) \Big)^{- 1} \xi\,. 
\end{equation}
Note that by the definition of $b_\alpha$ given in \eqref{trasporto per Egorov} and by the first equation in \eqref{Cauchy caratteristiche nel lemma}, one has that 
$$
\frac{d}{d \tau}\Big(x(\tau) +  \tau\alpha( x(\tau)) \Big) =  \alpha(x(\tau)) + \big( 1 + \tau \alpha_x(x(\tau)) \big) \dot x(\tau) = 0 
$$
implying that 
$$
x(\tau) + \tau \alpha(x(\tau)) = x + \tau_0 \alpha(x)\,, \quad  \forall \tau, \tau_0 \in [0, 1]\,. 
$$
In particular, for $\tau = 0$, one gets that 
$$
\gamma_1^{\tau_0, 0}(x) = x(0) =  x + \tau_0 \alpha(x) 
$$
and therefore by \eqref{hamsik 2} we obtain 
$$
\gamma_2^{\tau_0, 0}(x, \xi) =  (1 + \tau_0 \alpha_x(x))^{- 1}\xi\,,
$$
which proves the claimed statement. 
\end{proof}
Now, we are ready to state the Egorov Theorem.
\begin{theorem}\label{Teorema egorov generale}
Let $m \in \R$, ${\cal V}(t) = {\rm Op}\big( v(t, x, \xi)\big)$ be in the class $S^m$ and $\Phi(\tau; t)$, $\tau \in[0, 1]$ be the flow map of the PDE \eqref{equazione flusso operatoriale}. Then ${\cal P}(\tau; t) := \Phi(\tau; t) {\cal V}(t) \Phi(\tau; t)^{- 1}$ is a pseudo differential operator in the class $OPS^m$, i.e. ${\cal P}(\tau; t) = \Phi(\tau; t) {\cal V}(t) \Phi(\tau; t)^{-1} = {\rm Op}\big(  p(\tau; t, x, \xi) \big)$ with $p(\tau, \cdot, \cdot, \cdot) \in S^m$, $\tau \in [0, 1]$. Furthermore $p(\tau; t, x, \xi)$ admits the expansion 
$$
p(\tau; t, x, \xi) = p_0(\tau; t, x, \xi) + p_{\geq 1}(\tau; t, x, \xi)\,, \qquad p_0(\tau, \cdot, \cdot, \cdot) \in S^m\,, \quad p_{\geq 1}(\tau; \cdot, \cdot, \cdot) \in S^{m - 1}
$$
and the principal symbol $p_0$ has the form 
$$
\begin{aligned}
& p_0(\tau; t, x, \xi) := v\Big(t, x + \tau\alpha(t, x), (1 + \tau\alpha_x(t, x))^{- 1} \xi \Big)\,, \\
&   \qquad \forall (t, x, \xi) \in \R \times \T \times \R\,, \quad \forall \tau \in  [0, 1]\,. 
\end{aligned}
$$
\end{theorem}
\begin{proof}
We closely follow Theorem A.0.9 in \cite{Taylor}. A direct calculation shows that ${\cal P}(\tau; t)$ solves the Heisenberg equation 
\begin{equation}\label{heisenberg egorov astratto}
\begin{cases}
\partial_\tau {\cal P}(\tau; t) =[{\cal A}(\tau; t), {\cal P}(\tau; t)] \\
 {\cal P}(0; t) = {\cal V}(t)\,.
\end{cases}
\end{equation}
We then look for a solution ${\cal P}(\tau; t ) = {\rm Op}\Big( p(\tau; t, x, \xi) \Big) \in OPS^m$ with  
$$
p(\tau; t, x, \xi) \sim \sum_{n \geq 0} p_n(\tau; t, x, \xi)\,, \qquad p_n \in S^{m - n} \,, \quad \forall n \geq 0\,. 
$$
We show how to compute the asymptotic expansion of the symbol $p$. The operator ${\cal A}(\tau; t)$ in \eqref{trasporto per Egorov} has symbol
\begin{equation}\label{splitting simbolo cal A tau t}
\begin{aligned}
& a = \ii a_1 + a_0\,,  \\
& a_1(\tau; t, x, \xi) :=  b_\alpha(\tau; t, x) \xi \in S^1 \,, \quad a_0(\tau; t, x, \xi) := \frac{(\partial_x b_{\alpha})(\tau; t, x)}{2} \in S^0\,.
\end{aligned}
\end{equation}
 The symbol of the commutator $[{\cal A}(\tau; t), {\cal P}(\tau; t)] = {\rm Op}(a \star p)$ has the asymptotic expansion 
\begin{align}
g \star p & \sim \sum_{n \geq 0} a \star p_n\,.
\end{align}
Note that if $p_n \in S^{m - n}$, by \eqref{splitting simbolo cal A tau t} and Corollary \ref{corollario commutator}, one has that 
\begin{equation}
\begin{aligned}
& \ii a_1 \star p_n  = \{ a_1, p_n \} + \ii \mathtt r_2( a_1, p_n), \quad  \{ a_1, p_n \} \in S^{m - n}\,, \quad \mathtt r_2( a_1, p_n) \in S^{m - n - 1} \\
& a_0 \star p_n \in S^{m - n - 1}\,.
\end{aligned}
\end{equation}
This implies that 
$$
a \star p_n = \{ a_1, p_n \} + q_n\,, \quad q_n := \ii \mathtt r_2( a_1, p_n) + a_0 \star p_n \in S^{m - n - 1}\,. 
$$
We then solve iteratively 
\begin{equation}\label{equazione p0 A}
\begin{cases}
\partial_\tau p_0 = \{a_1, p_0 \} \\
p_0(0; t, x, \xi) = v(t, x, \xi)
\end{cases}
\end{equation}
and 
\begin{equation}\label{equazione pn A}
\begin{cases}
\partial_\tau p_n (\tau; t, x, \xi) = \{a_1, p_n \} + q_{n - 1}\\
p_n(0; t, x, \xi) = 0\,, 
\end{cases}\quad \forall n \geq 1\,.
\end{equation}
Using the characteristic method, the solutions of \eqref{equazione p0 A}, \eqref{equazione pn A} are given by 
\begin{equation}\label{soluzione p0 A}
p_0(\tau; t, x, \xi) := v(t, \gamma_1^{\tau, 0}(t, x), \gamma_2^{\tau, 0}(t, x, \xi))\,, \quad \forall \tau \in [0, 1]
\end{equation}
and 
\begin{equation}\label{soluzione pn A}
p_n(\tau; t, x, \xi) =  \int_0^\tau q_{n - 1}(\zeta; t, \gamma_1^{\tau, \zeta}(t, x), \gamma_2^{\tau, \zeta}(t, x, \xi))\, d \zeta\,, \qquad \forall n \geq 1\,,
\end{equation}
where for any $\tau, \zeta \in [0, 1]$, $(\gamma_1^{\tau, \zeta}, \gamma_2^{\tau, \zeta})$ is the flow of the ODE \eqref{flusso caratteristiche}. The claimed statement then follows by applying Lemma \ref{espressione esplicita caratteristiche} and by setting $p_{\geq 1} \sim \sum_{n \geq 1} p_n \in OPS^{m - 1}$. 
\end{proof}
In the following we will also need to analyse the operator $\Phi(\tau; t) \partial_t (\Phi(\tau; t)^{- 1})$. The following lemma holds: 
\begin{theorem}\label{parte tempo egorov}
The operator $\Psi(\tau; t) = \Phi(\tau; t) \partial_t (\Phi(\tau; t)^{- 1})$, $\tau \in [0, 1]$ is a pseudo differential operator in the class $OPS^1$. 
\end{theorem}
\begin{proof}
First we compute $\partial_\tau \Psi(\tau; t)$. One has  
\begin{align}
\partial_\tau \Psi(\tau; t) & = \partial_\tau \Phi(\tau; t ) \partial_t \big( \Phi(\tau; t)^{- 1}  \big)+ \Phi(\tau; t) \partial_t \partial_\tau\big(\Phi(\tau; t)^{- 1} \big) \nonumber\\
& =  \partial_\tau \Phi(\tau; t) \partial_t \big( \Phi(\tau; t)^{- 1} \big) -  \Phi(\tau; t) \partial_t \Big(\Phi(\tau; t)^{- 1} \partial_\tau \Phi(\tau; t) \Phi(\tau; t)^{- 1} \Big)  \nonumber\\
& \stackrel{\eqref{equazione flusso operatoriale}}{=}  {\cal A}(\tau; t) \Phi(\tau; t) \partial_t \Phi(\tau; t)^{- 1}  - \Phi(\tau; t) \partial_t \Big(\Phi(\tau; t)^{- 1} {\cal A}(\tau; t)  \Big) \nonumber\\
& =  [{\cal A}(\tau; t), \Psi(\tau; t)]  - \partial_t {\cal A}(\tau; t) \,. 
\end{align}
Therefore $\Psi(\tau; t)$ solves 
\begin{equation}\label{equazione Psi pm M = 1}
\begin{cases}
\partial_\tau \Psi(\tau; t)  =  [{\cal A}(\tau; t), \Psi(\tau; t)]  -  \partial_t {\cal A}(\tau; t) \\
\Psi(0; t) = 0\,. 
\end{cases}
\end{equation}
Arguing as in Theorem \ref{Teorema egorov generale}, we find that $\Psi(\tau; t) = {\rm Op}\Big( \psi(\tau; t, x, \xi) \Big) \in OPS^1$, by solving \eqref{equazione Psi pm M = 1} in decreasing orders and by determing an asymptotic expansion of the symbol $\psi$ of the form 
$$
\psi \sim \sum_{n \geq 0}\psi_{ n}\,, \quad \psi_{n} \in S^{1 - n}\,, \quad \forall n \geq 0\,. 
$$
\end{proof}
We also state another {\it semplified} version of the Egorov theorem in which we conjugate a symbol by means of the flow of a vector field which is a pseudo differential operator of order strictly smaller than one. 
We consider a pseudo differential operator ${\cal G}(t) = {\rm Op}(g(t, x, \xi))$, with $g \in S^{\eta}$, ${\cal G}(t) = {\cal G}(t)^*$, $\eta < 1$ and for any $\tau \in [0, 1]$, let $\Phi_{{\cal G}}(\tau; t)$ be the flow of the pseudo-PDE 
\begin{equation}\label{pseudo ordine < 1}
\partial_\tau u = \ii {\cal G}(t) u\,,
\end{equation}
which is a well-defined invertible map by Lemma \ref{flusso + derivate flusso}. 
Then $\Phi_{{\cal G}}(\tau; t)$ solves 
\begin{equation}\label{flusso pseudo ordine < 1}
\begin{cases}
\partial_\tau \Phi_{{\cal G}}(\tau; t) = \ii {\cal G}(t) \Phi_{{\cal G}}(\tau; t) \\
\Phi_{{\cal G}}(0; t)  = {\rm Id}\,.
\end{cases}
\end{equation}
The following theorem holds. 
\begin{theorem}\label{teorema egorov campo minore di uno}
Let $m \in \R$, ${\cal V}(t) = {\rm Op}\big( v(t, x, \xi) \big) \in OPS^m$ and ${\cal G}(t) = {\rm Op}(g(t, x, \xi))$, with $g \in S^{\eta}$, $\eta < 1$. Then for any $\tau \in [0, 1]$, the operator ${\cal P}(\tau; t) := \Phi_{{\cal G}}(\tau; t) {\cal V}(t) \Phi_{{\cal G}}(\tau; t)^{- 1}$ is a pseudo differential operator of order $m$ with symbol $p(\tau; \cdot , \cdot, \cdot ) \in S^m$. The symbol $p(\tau; t , x, \xi )$ admits the expansion 
\begin{equation}\label{espansione lemma egorov semplificato}
p(\tau; t , x, \xi ) = v(t, x, \xi) + \tau \{ g, v \}(t, x, \xi) + p_{\geq 2}(\tau; t, x, \xi)\,, \quad p_{\geq 2}(\tau; t, x, \xi) \in S^{m - 2(1 - \eta)}\,. 
\end{equation}
\end{theorem}
\begin{proof}
 We show how to compute the asymptotic expansion of the operator ${\cal P}(\tau; t)$ by taking advantage from the fact that the order of ${\cal G}(t)$ is strictly smaller than $1$. A direct calculation shows that ${\cal P}(\tau; t)$ solves the Heisenberg equation 
\begin{equation}\label{heisenberg egorov astratto}
\begin{cases}
\partial_\tau {\cal P}(\tau; t) = \ii [{\cal G}(t), {\cal P}(\tau; t)] \\
 {\cal P}(0; t) = {\cal V}(t)\,.
\end{cases}
\end{equation}
We then look for ${\cal P}(\tau; t ) = {\rm Op}\Big( p(\tau; t, x, \xi) \Big) \in OPS^m$ with  
$$
p(\tau; t, x, \xi) \sim \sum_{n \geq 0} p_n(\tau; t, x, \xi)\,, \qquad p_n(\tau; t, x, \xi) \in S^{m - n (1 - \eta)} \,, \quad \forall n \geq 0\,. 
$$
The symbol of the commutator $[{\cal G}(t), {\cal P}(\tau; t)] = {\rm Op}(g \star p)$ has the asymptotic expansion 
\begin{align}
g \star p & \sim \sum_{n \geq 0} g \star p_n 
\end{align}
Note that if $p_n \in S^{m - n (1 - \eta)}$, by Corollary \ref{corollario commutator}, one has that $g \star p_n \in S^{m - (n + 1)(1 - \eta)}$. We then solve iteratively 
\begin{equation}\label{equazione p0}
\begin{cases}
\partial_\tau p_0 (\tau; t, x, \xi) = 0 \\
p_0(0; t, x, \xi) = v(t, x, \xi)
\end{cases}
\end{equation}
and 
\begin{equation}\label{equazione pn}
\begin{cases}
\partial_\tau p_n (\tau; t, x, \xi) = \ii g \star p_{n - 1} \\
p_n(0; t, x, \xi) = 0\,, 
\end{cases}\quad \forall n \geq 1\,.
\end{equation}
The solutions of \eqref{equazione p0}, \eqref{equazione pn} are then given by 
\begin{equation}\label{soluzione p0}
p_0(\tau; t, x, \xi) := v(t, x, \xi)\,, \quad \forall \tau \in [0, 1]
\end{equation}
and 
\begin{equation}\label{soluzione pn}
p_n(\tau; t, x, \xi) = \ii \int_0^\tau g \star p_{n - 1} (\zeta; t, x, \xi)\, d \zeta\,, \qquad \forall n \geq 1\,.
\end{equation}
In order to determine the expansion \eqref{espansione lemma egorov semplificato}, we analyze the symbol $p_1$. By \eqref{soluzione p0}, \eqref{soluzione pn}, one gets 
$$
p_1 (\tau; t, x, \xi) = \ii \tau g \star v(t, x, \xi) \stackrel{Corollary\, \ref{corollario commutator}}{=} \tau \{ g, v \} + \ii \tau \mathtt r_2(g, v)
$$
and 
$$
\mathtt r_2(g, v) \in S^{m + \eta - 2} \stackrel{\eqref{inclusioni Sm OPSm}}{\subseteq} S^{m - 2(1 - \eta)}
$$
therefore the expansion \eqref{espansione lemma egorov semplificato} is determined by taking 
$$
p_{\geq 2}(\tau; t, x, \xi) \sim \ii \tau \mathtt r_2(g, v) + \sum_{n \geq 2} p_n(\tau; t, x, \xi)\,. 
$$
\end{proof}

\section{Regularization of the vector field ${\cal V}(t)$}\label{sezione regolarizzazione cal V (t)}
In this section we develop the regularization procedure on the vector field $\ii {\cal V}(t) = \ii \big(V(t, x) |D|^M + {\cal W}(t) \big)$, see \eqref{forma iniziale cal V (t)}, which is needed to prove Theorem \ref{teorema riduzione}. In Section \ref{riduzione ordine principale M > 1} we reduce to constant coefficients the highest order $V(t, x) |D|^M$, see Proposition \ref{teorema riassunto primo egorov}. Then, in Section \ref{Reduction of the lower order terms}, we perform the reduction of the lower order terms up to arbitrarily regularizing remainders, see Proposition \ref{descent method M geq 1}. 

\subsection{Reduction of the highest order}\label{riduzione ordine principale M > 1}
Our first aim is to eliminate the $x$-dependence from the highest order of the vector field $\ii{\cal V}(t)$, namely we want to eliminate the $x$-dependence from the term $V(t, x) |D|^M$. To this aim, let us consider a ${\cal C}^\infty$ function $\alpha : \R \times \T \to \R$ (that will be fixed later) satisfying the following ansatz: 
\begin{equation}\label{ansatz alpha egorov}
\begin{aligned}
\alpha \in {\cal C}^\infty_b(\R \times \T, \R)\,, \quad  \inf_{(t, x) \in \R \times \T} \big(1 + \alpha_x(t, x) \big) > 0\,. 
\end{aligned}
\end{equation}
Then, we consider the non-autonomous transport equation
\begin{equation}\label{trasporto per Egorov}
\begin{aligned}
& \partial_\tau u = {\cal A}(\tau; t, x, D) [u]\,, \\
& {\cal A}(\tau; t) = {\cal A}(\tau;t,   x, D) := b_{\alpha}(\tau; t, x) \partial_x + \frac{(\partial_x b_{\alpha})(\tau; t, x)}{2}\,,  \\
& b_\alpha (\tau; t, x):= - \frac{\alpha(t, x)}{1 + \tau \alpha_x(t, x)} \,, \qquad \tau \in [0, 1]\,. 
\end{aligned}
\end{equation}  
By Lemma \ref{proprieta flusso trasporto lemma astratto}, the flow $\Phi(\tau; t)$, $\tau \in [0, 1]$ of the equation \eqref{trasporto per Egorov}, i.e. 
\begin{equation}\label{equazione flusso operatoriale}
\begin{cases}
\partial_\tau \Phi(\tau; t) = {\cal A}(\tau; t) \Phi(\tau; t)  \\
\Phi (0; t) = {\rm Id}
\end{cases}
\end{equation}
is a well-defined, symplectic, invertible map $H^s \to H^s$ for any $s \in\R$. We define $\Phi (t) := \Phi(1 ; t)$. In order to state the Proposition below, we introduce the constant 
 \begin{equation}\label{definizione costante frak e1}
\bar{\frak e}:= M - {\rm max}\{M - 1, 1 , M - \frak e \}\,.
\end{equation}  
Note that, by the above definition and using that $M > 1$, it follows easily that  
\begin{equation}\label{proprieta costantina frak e1}
\bar{\frak e} > 0\,, \quad M - \bar{\frak e} \geq M -1\,,\, 1\,,\, M - \frak e\,. 
\end{equation}
 \begin{proposition}\label{teorema riassunto primo egorov}
The symplectic invertible map $\Phi(t) = \Phi(1; t)$, given by \eqref{equazione flusso operatoriale}, satisfies 
\begin{equation}\label{proprieta Phi 2}
\sup_{t \in \R} \| \Phi(t)^{\pm 1}\|_{{\cal B}(H^s)} + \sup_{t \in \R} \| \partial_t \Phi(t)^{\pm 1}\|_{{\cal B}(H^{s + 1}, H^{s})} < + \infty\,, \quad \forall s  \geq 0\,. 
\end{equation}
There exist a function $ \lambda \in {\cal C}^\infty_b (\R, \R)$ satisfying  
\begin{equation}\label{proprieta lambda alpha tilde primo egorov}
 \inf_{t \in \R} \lambda(t) > 0
\end{equation}
and an operator 
$$
{\cal W}_1(t) = {\rm Op}\Big( w_1(t, x, \xi)\Big), \quad w_1 \in S^{M - \bar{\frak e}}
$$
with ${\cal W}_1(t) = {\cal W}_1(t)^*$ for any $t \in \R$, such that 
\begin{equation}\label{push forward cal V 1 teorema}
(\Phi^{- 1})_*(\ii {\cal V})(t) = \ii {\cal V}_1(t) \quad \text{with} \quad {\cal V}_1(t) := \lambda(t) |D|^M + {\cal W}_1(t)\,. 
\end{equation}
\end{proposition}
All the rest of this section is devoted to the proof of the proposition stated above. The property \eqref{proprieta Phi 2} follows by applying Lemma \ref{flusso + derivate flusso}, using that ${\cal A}(\tau; t) \in OPS^1$ and using that, by a direct calculation, ${\cal A}(\tau; t) + {\cal A}(\tau; t)^* = 0$.  The push-forward of the vector field $\ii {\cal V}(t)$ by means of the map $\Phi(t)^{- 1}$ is then given by $\ii {\cal V}_1(t)$ with  
\begin{align}
{\cal V}_1(t) & =  \Phi(t) {\cal V}(t) \Phi(t)^{- 1} + \ii \Phi(t) \partial_t \big( \Phi(t)^{- 1} \big)\,.  \label{nuovo campo cal V1}
\end{align}
By applying Lemmata \ref{Teorema egorov generale}, \ref{parte tempo egorov}, one has that ${\cal V}_1(t) = {\rm Op}(v_1(t, x, \xi))\in OPS^M$ with 
\begin{equation}\label{bomba di cavani}
v_1(t, x, \xi) = p_0(t, x, \xi) + p_{\geq 1}(t, x, \xi)\,,
\end{equation}
with 
\begin{equation}\label{bomba di cavani 1}
\begin{aligned}
& p_0(t, x, \xi) := v\Big(t, x + \alpha(t, x), (1 + \alpha_x(t, x))^{- 1} \xi \Big)\,, \quad p_{\geq 1} \in S^{{\rm max}\{M - 1, 1\}} \stackrel{\eqref{inclusioni Sm OPSm}, \eqref{proprieta costantina frak e1}}{\subseteq} S^{M - \bar{\frak e}} \,. 
\end{aligned}
\end{equation}
In the next lemma we compute the expansion of the symbol $v_1(t, x, \xi)$ of the operator ${\cal V}_1(t)$ defined in \eqref{nuovo campo cal V1}. 

\begin{lemma}\label{espansione simbolo principale egorov}
The symbol $v_1(t, x, \xi)$ has the form 
\begin{equation}\label{espansione simbolo q0 trasporto}
v_1(t, x, \xi) = \Big[V(t, y) \big( 1 + \widetilde \alpha_y(t, y) \big)^M \Big]_{y = x + \alpha(t, x)} |\xi|^M \chi(\xi) + w_{1}(t, x, \xi)\,, \quad w_1 \in S^{M - \bar{\frak e}}\,
\end{equation}
where we recall the definitions \eqref{definizione cut off D alpha}, \eqref{definizione D alpha} and $y \mapsto y + \widetilde \alpha(t, y)$ is the inverse diffeomorphism of $x \mapsto x + \alpha(t, x)$.   
\end{lemma}
\begin{proof}
Using \eqref{nuovo campo cal V1}-\eqref{bomba di cavani 1}, one obtains
\begin{align}
v_1(t, x, \xi) & = p_0(t, x, \xi) + p_{\geq 1}(t, x, \xi)  \nonumber\\
& =  v\Big(t, x + \alpha(t, x), (1 + \alpha_x(t, x))^{- 1} \xi \Big) + p_{\geq 1}(t, x, \xi)     \label{prima espansion v1} 
\end{align}
 By \eqref{forma iniziale cal V (t)}, the symbol of the operator ${\cal V}(t)$ has the form 
$$
v(t, x, \xi) =  V(t, x) |\xi|^M \chi(\xi) + w(t, x, \xi)\,, \quad w \in S^{M - \frak e}\,,
$$
hence, by \eqref{bomba di cavani 1}, one has  
\begin{align}
p_0(t, x, \xi) & =  v \Big( t, x + \alpha(t, x), (1 + \alpha_x(t, x))^{- 1} \xi \Big)  \nonumber\\
& = V(t, x + \alpha(t, x)) (1 + \alpha_x(t, x))^{- M} |\xi|^M \chi\Big( (1 + \alpha_x(t, x))^{- 1} \xi \Big)  \nonumber\\
& \quad + w_{p_0}(t, x, \xi) \label{prima espansione q0}
\end{align}
where 
\begin{equation}\label{grado w p0}
w_{p_0}(t, x, \xi) := w\Big(t, x + \alpha(t, x), (1 + \alpha_x(t, x))^{- 1} \xi \Big) \in S^{ M - \frak e} \stackrel{\eqref{proprieta costantina frak e1}}{\subseteq} S^{M - \bar{\frak e}}. 
\end{equation}
By using the mean value theorem, one writes 
$$
\chi\Big( (1 + \alpha_x(t, x))^{- 1} \xi \Big) = \chi(\xi) + w_\chi(t, x, \xi)\,, \quad  
$$
$$
w_\chi(t, x, \xi) :=- \frac{\alpha_x(t, x) \xi}{1 + \alpha_x(t, x)} \int_0^1 \partial_\xi \chi\Big( \zeta  (1 + \alpha_x(t, x))^{- 1} \xi + (1 - \zeta) \xi\Big)\, d \zeta\,. 
$$
Since $\partial_\xi \chi(\xi) = 0$ for any $|\xi| \geq 1$ (see \eqref{definizione cut off D alpha}) one has that the symbol $w_\chi \in OPS^{- \infty}$, hence by \eqref{prima espansion v1}, \eqref{prima espansione q0} one obtains that 
$$
v_1(t, x, \xi) = V(t, x + \alpha(t, x)) (1 + \alpha_x(t, x))^{- M} |\xi|^M \chi(\xi) + w_{1}(t, x, \xi)
$$
where 
\begin{align*}
w_{1}(t, x, \xi) & := p_{\geq 1} (t, x, \xi) + w_{p_0} (t, x, \xi) \\
& \quad + V(t, x + \alpha(t, x)) (1 + \alpha_x(t, x))^{- M} |\xi|^M w_\chi(t, x, \xi)\,. 
\end{align*}
Recalling \eqref{bomba di cavani 1}, \eqref{grado w p0} and that $w_\chi \in S^{- \infty}$, one obtains that $w_1 \in S^{M - \bar{\frak e}}$. 
Since $\alpha$ satisfies \eqref{ansatz alpha egorov}, we can apply lemma \ref{diffeo del toro lemma astratto}, obtaining that the diffeomorphism of the torus $x \mapsto x + \alpha(t, x)$ is invertible with inverse $y \mapsto y + \widetilde \alpha(t, y)$ and $\widetilde \alpha \in {\cal C}^\infty_b (\R \times \T, \R)$ satisfies \eqref{proprieta alpha tilde diffeo toro}.  Using the identity \eqref{identita 1 + alpha alpha tilde partial}, we then have  
$$
V(t, x + \alpha(x)) (1 + \alpha_x(x))^{- M} = \Big[V(t, y) \big( 1 + \widetilde \alpha_y(y) \big)^M \Big]_{y = x + \alpha(t, x)}
$$
and the lemma is proved. 
\end{proof}
We now determine the function $\widetilde \alpha(t, y)$ so that 
\begin{equation}\label{equazione omologica grado alto}
V(t, y) \big( 1 + \widetilde \alpha_y(t, y) \big)^M = \lambda(t)\,,
\end{equation}
for some bounded and real-valued ${\cal C}^\infty$ function $\lambda$, to be determined. 
The equation \eqref{equazione omologica grado alto} is equivalent to the equation
\begin{equation}\label{equazione omologica grado alto 2}
\widetilde \alpha_y(t, y) = \frac{\lambda(t)^{\frac{1}{M}}}{V(t, y)^{\frac{1}{M}}} - 1\,.
\end{equation}
Notice that, by the assumption {\bf (H2)}, $V(t, y)$ does never vanish. We choose $\lambda(t)$ so that the average of the right hand side of the equation \eqref{equazione omologica grado alto 2} is $0$, hence we set
\begin{equation}\label{lambda 1 (t)}
\lambda(t) :=  \Big(\frac{1}{2 \pi}\int_\T V(t, y)^{- \frac{1}{M}}\, d y \Big)^{- M}\,.
\end{equation}
Therefore, we solve \eqref{equazione omologica grado alto 2} by defining 
\begin{equation}\label{definizione widetilde alpha}
\widetilde \alpha(t, y) := \partial_y^{- 1} \Big[ \frac{\lambda_1(t)^{\frac{1}{M}}}{V(t, y)^{\frac{1}{M}}} - 1 \Big] 
\end{equation}
(recall the definition \eqref{definizione partial x - 1}). Note that by the hypothesis {\bf (H2)} on $V$ and by the definitions \eqref{lambda 1 (t)}, \eqref{definizione widetilde alpha}, one has $\lambda \in {\cal C}^\infty_b(\R, \R)$, $\widetilde \alpha\in {\cal C}^\infty_b(\R \times \T, \R)$ and   
$$
\inf_{t \in \R} \lambda(t) > 0\,, \quad \inf_{(t, y) \in \R \times \T}\Big(1 + \widetilde \alpha_y(t, y) \Big) > 0
$$
which verifies \eqref{proprieta lambda alpha tilde primo egorov}. Then by applying lemma \ref{diffeo del toro lemma astratto} one gets that the function $\alpha$ satisfies the ansatz \eqref{ansatz alpha egorov} since $x \mapsto x + \alpha(t, x)$ is the inverse diffeomorphism of $y \mapsto y + \widetilde \alpha(t, y)$. 

\noindent
Finally, by lemma \ref{espansione simbolo principale egorov} and since $\widetilde \alpha$ and $\lambda$ solve the equation \eqref{equazione omologica grado alto}, we obtain that ${\cal V}_1(t)$ is given by 
\begin{equation}\label{forma finale cal V1 (t)}
{\cal V}_1(t) = \lambda(t) |D|^M + {\cal W}_1(t)\,, \quad {\cal W}_1(t) = {\rm Op}\big( w_1(t, x, \xi) \big) \in OPS^{M - \bar{\frak e}}\,. 
\end{equation}
Since $\Phi(t)$ is symplectic, the vector field $\ii {\cal V}_1(t)$ is Hamiltonian, i.e. ${\cal V}_1(t)$ is $L^2$ self-adjoint. Since $\lambda(t)|D|^M$ is selfadjoint, then ${\cal W}_1(t) = {\cal V}_1(t) - \lambda(t) |D|^M $ is self-adjoint too, hence the proof of Proposition \ref{teorema riassunto primo egorov} is concluded. 
\subsection{Reduction of the lower order terms}\label{Reduction of the lower order terms}
In this Section we transform the vector field $\ii {\cal V}_1(t)$, obtained in Proposition \ref{teorema riassunto primo egorov}, into another one which is an arbitrarily regularizing perturbation of a {\it space-diagonal} operator. This is done in the following   
\begin{proposition}\label{descent method M geq 1}
Let $N \in \N$. For any $n = 1, \ldots, N$ there exists a linear Hamiltonian vector field $\ii {\cal V}_n(t)$ of the form 
\begin{equation}\label{forma cal Vn teorema}
{\cal V}_n(t) := \lambda(t) |D|^M + \mu_n(t, D) + {\cal W}_n(t)\,,
\end{equation}
where 
\begin{equation}\label{mu n t D teo}
\mu_n(t, D) := {\rm Op}\Big( \mu_n(t, \xi)\Big)\,, \qquad \mu_n \in S^{M - \bar{\frak e}}\,,
\end{equation}
\begin{equation}\label{cal Wn teorema}
{\cal W}_n(t) := {\rm Op}\Big( w_n(t, x, \xi)\Big) \,, \qquad w_n \in S^{M - n \bar{\frak e}}\,,
\end{equation}
with $\mu_n(t, \xi)$ real and ${\cal W}_n(t)$ $L^2$ self-adjoint, i.e. $w_n = w_n^*$ (see Theorem \ref{adjoint}).

\noindent
For any $n \in \{ 1, \ldots, N - 1\}$, there exists a symplectic invertible map $\Phi_n(t)$ satisfying
\begin{equation}\label{proprieta Phi n teo 1}
\sup_{t \in \R}\| \Phi_n(t)^{\pm 1} \|_{{\cal B}(H^{s})} + \sup_{t \in \R}\| \partial_t \Phi_n (t)^{\pm 1}\|_{{\cal B}(H^{s + 1}, H^s)}  < + \infty\,, \quad \forall s \geq 0 
\end{equation}
and 
\begin{equation}\label{cal V n + 1 cal V n}
\ii {\cal V}_{n + 1}(t) = (\Phi_n^{- 1})_*( \ii {\cal V}_n)(t)\,, \qquad \forall n \in \{ 1, \ldots, N - 1 \}\,. 
\end{equation}
\end{proposition}
The rest of the section is devoted to the proof of the above Proposition. It is proved arguing by induction. Let us describe the induction step. At the $n$-th step , we deal with a Hamiltonian vector field of the form $\ii {\cal V}_n(t)$ which satisfies the properties \eqref{forma cal Vn teorema}-\eqref{cal Wn teorema}. We look for an operator ${\cal G}_n(t)$ of the form 
\begin{equation}\label{ordine cal G1}
{\cal G}_n(t) := {\rm Op}\big( g_n(t, x, \xi) \big) \in OPS^{1 - n \bar{\frak e} } \quad \text{with} \quad  {\cal G}_n(t) = {\cal G}_n(t)^*
\end{equation}
and we consider the flow $\Phi_{{\cal G}_n}(\tau; t)$ of the pseudo PDE 
\begin{equation}\label{flusso Phi cal G1}
\partial_\tau u = \ii {\cal G}_n(t)[u] \,. 
\end{equation}
The flow map $\Phi_{{\cal G}_n}(\tau; t)$ solves 
\begin{equation}\label{flusso Phi cal G1 (t)}
\begin{cases}
\partial_\tau \Phi_{{\cal G}_n}(\tau; t) = \ii {\cal G}_n(t) \Phi_{{\cal G}_n}(\tau; t) \\
\Phi_{{\cal G}_n}(0; t) = {\rm Id}\,. 
\end{cases}
\end{equation}
Note that, since ${\cal G}_n(t)$ is self-adjoint, $ \ii {\cal G}_n(t)$ is a Hamiltonian vector field, implying that $\Phi_{{\cal G}_n}(\tau; t)$ is symplectic for any $\tau \in [0, 1]$, $t \in \R$. Since ${\cal G}_n(t) \in OPS^{1 - n \bar{\frak e}} \stackrel{\eqref{inclusioni Sm OPSm}}{\subseteq} OPS^1$ and $(\ii {\cal G}_n(t)) + (\ii {\cal G}_n(t))^* = \ii \big({\cal G}_n(t) - {\cal G}_n(t)^* \big) = 0$, by Lemma \ref{flusso + derivate flusso}, the maps $\Phi_{{\cal G}_n}(\tau; t)^{\pm 1}$ satisfy the property \eqref{proprieta Phi n teo 1}. Note that, since the vector field ${\cal G}_n(t)$ does not depend on $\tau$, one has $\Phi_{{\cal G}_n}(\tau; t)^{- 1} = \Phi_{{\cal G}_n}(- \tau; t)$. We set $\Phi_{n}(t) := \Phi_{{\cal G}_n}(1; t)$. The transformed vector field is given by $(\Phi_{n}^{- 1})_* (\ii {\cal V}_n)(t) = \ii {\cal V}_{n + 1}(t)$, where  
\begin{equation}\label{prima definizione cal V2}
{\cal V}_{n + 1}(t) := \Phi_{n}(t) {\cal V}_n(t) \Phi_{n}(t)^{- 1} + \ii \Phi_{n}(t) \partial_t \big( \Phi_{n}(t)^{- 1} \big)\,.
\end{equation}
Since ${\cal G}_n(t)$ is a pseudo-differential operator of order strictly smaller than $1$, we can apply Theorem \ref{teorema egorov campo minore di uno}, obtaining that ${\cal P}_n(t) = {\rm Op}\big(p_n(t, x, \xi) \big) := \Phi_{n}(t) {\cal V}_n(t) \Phi_{n}(t)^{- 1} \in OPS^M$ with 
\begin{equation}\label{prima espansion pn}
p_n = v_n +  \{ g_n , v_n \} + p_{n, \geq 2}\,, \quad p_{n, \geq 2} \in S^{M - 2 n \bar{\frak e}} \stackrel{\eqref{inclusioni Sm OPSm}}{\subseteq} S^{M - (n + 1)  \bar{\frak e}}\,. 
\end{equation}
Furthermore, defining $\Psi_n(\tau; t) := \ii  \Phi_{{\cal G}_n}(\tau; t) \partial_t \big( \Phi_{{\cal G}_n}(\tau; t)^{- 1} \big)$, a direct calculation shows that 
$$
 \Psi_n(\tau; t) = \ii \int_0^\tau {\cal S}_{{\cal G}_n}(\zeta; t)\, d \zeta\,, \quad  {\cal S}_{{\cal G}_n}(\zeta; t) := \Phi_{{\cal G}_n}(\zeta; t) \partial_t {\cal G}_n( t) \Phi_{{\cal G}_n}(\zeta; t)^{- 1}. 
$$
Since $\partial_t {\cal G}_n(t) \in OPS^{1 - n  \bar{\frak e}}$, by Theorem \ref{teorema egorov campo minore di uno} 
$$
\Psi_n(t) \equiv \Psi_n (1 ; t) = \ii \Phi_n( t) \partial_t \big( \Phi_n( t)^{- 1} \big) = {\rm Op}\Big(\psi_n(t, x, \xi) \Big) \in OPS^{1 - n  \bar{\frak e}}\,. 
$$
Using that $M -  \bar{\frak e} \geq 1$ (see \eqref{definizione costante frak e1}), one gets that 
\begin{equation}\label{ordine Psi n}
\Psi_n (t) = {\rm Op}\big( \psi_n(t, x, \xi)\big) \in OPS^{1 - n \bar{\frak e}} \stackrel{\eqref{inclusioni Sm OPSm}}{\subseteq} OPS^{M - (n + 1) \bar{\frak e}}\,. 
\end{equation}
In the next lemma, we provide an expansion of the symbol $v_{n + 1}(t, x, \xi)$ of the operator ${\cal V}_{n + 1}(t)$ given in \eqref{prima definizione cal V2}. 
\begin{lemma}\label{prima espansione simbolo v n+1}
The operator ${\cal V}_{n + 1}(t) = {\rm Op}\big( v_{n + 1}(t, x, \xi) \big) \in OPS^{M}$ admits the expansion 
\begin{align}
v_{n + 1}(t, x, \xi) & = \lambda(t) |\xi|^M \chi(\xi) + \mu_n(t, \xi)+ w_n(t, x, \xi) \nonumber\\
& \quad - M \lambda(t) |\xi|^{M - 2} \xi \chi(\xi)  (\partial_x g_n)(t, x, \xi) + r_{v_{n}}(t, x, \xi) \label{espansione v n + 1 per eq omologica}
\end{align}
where $r_{v_n} \in S^{M - (n + 1) \bar{\frak e}}$. 
\end{lemma}
\begin{proof}
By \eqref{prima definizione cal V2}-\eqref{ordine Psi n}, one has 
\begin{align}
v_{n + 1} & = v_n +  \{ g_n , v_n \}+ p_{n, \geq 2} + \psi_n\,. \label{seconda espansione v n + 1}
\end{align}
Since, by the induction hypothesis,
$$
v_{n }(t, x, \xi) = \lambda(t) |\xi|^M \chi(\xi) + \mu_n(t, \xi) + w_n(t, x, \xi)\,, \quad \mu_n \in S^{M - \bar{\frak e}}\,, \quad w_n \in S^{M - n  \bar{\frak e}}
$$
one has that 
\begin{align}
\{ g_n , v_n \}& = \{ g_n , \lambda(t) |\xi|^M \chi(\xi) \}  + \{ g_n , \mu_n \}  + \{ g_n , w_n \}  \nonumber\\
& = - \lambda(t) \partial_\xi \Big(   |\xi|^M \chi(\xi)  \Big) (\partial_x g_n) +  \{ g_n , \mu_n \}  + \{ g_n , w_n \} \nonumber\\
& = - M \lambda(t)     |\xi|^{M - 2} \xi  \chi(\xi)   (\partial_x g_n) - \lambda(t)     |\xi|^M  (\partial_\xi\chi(\xi))   (\partial_x g_n)  \nonumber\\
& \qquad +  \{ g_n , \mu_n \}  + \{ g_n , w_n \}\,. \label{espansione poisson nella riducibilita}
\end{align}
Using that $\partial_\xi \chi(\xi) = 0$ for $|\xi| \geq 1$, since $g_n \in S^{1 - n  \bar{\frak e}}$, $\mu_n \in S^{M -  \bar{\frak e}}$, $w_n \in S^{M - n  \bar{\frak e}}$, by Corollary \ref{corollario commutator} one gets 
\begin{equation}\label{espansione poisson nella riducibilita 0}
\lambda(t)     |\xi|^M  (\partial_\xi\chi(\xi))   (\partial_x g_n) \in S^{- \infty}\,, \quad \{g_n , \mu_n \} \in S^{M - (n + 1)  \bar{\frak e}}\,,
\end{equation}
\begin{equation}\label{espansione poisson nella riducibilita 1}
\{ g_n, w_n \} \in S^{M - 2 n  \bar{\frak e}} \stackrel{\eqref{inclusioni Sm OPSm}}{\subseteq} S^{M - (n + 1)  \bar{\frak e}}\,. 
\end{equation}
Thus, \eqref{seconda espansione v n + 1}, \eqref{espansione poisson nella riducibilita} imply the claimed expansion with 
$$
r_{v_n} :=  - \lambda(t)     |\xi|^M  (\partial_\xi\chi(\xi))   (\partial_x g_n) + \{ g_n , \mu_n \}  + \{ g_n , w_n \} + p_{n, \geq 2} + \psi_n\,. 
$$ 
Finally, \eqref{prima espansion pn}, \eqref{ordine Psi n}, \eqref{espansione poisson nella riducibilita 0}, \eqref{espansione poisson nella riducibilita 1} imply that $r_{v_n} \in S^{M - (n + 1) \bar{\frak e}}$. 
\end{proof}

\noindent
{\sc Choice of the symbol $g_n$.}
In the next lemma, we show that the symbol $g_n$ can be chosen in order to eliminate the $x$-dependence from the term of order $M - n  \bar{\frak e}$ in the expansion \eqref{espansione v n + 1 per eq omologica}.
\begin{lemma}\label{equazione omologica ordini bassi}
There exists a symbol $g_n \in S^{1 - n \bar{\frak e}}$, $g_n = g_n^*$, such that 
\begin{equation}\label{equazione omologica simbolo g1}
- \lambda(t)  M \, |\xi|^{M - 2} \xi \chi(\xi) (\partial_x g_n)(t, x, \xi) + w_n(t, x, \xi) -  \langle w_n\rangle_x (t, \xi)  \in S^{M - (n + 1) \bar{\frak e} }
\end{equation}
(recall the definition \eqref{simbolo mediato}). 
\end{lemma}
\begin{proof}
Let $\chi_0 \in {\cal C}^\infty(\R, \R)$ be a cut-off function satisfying 
\begin{equation}\label{scelta cut off chi 0}
\begin{aligned}
&\chi_0(\xi) = 1 \,, \quad \forall |\xi| \geq 2\,, \\
&\chi_0(\xi) = 0\,, \quad \forall |\xi| \leq 1\,. 
\end{aligned}
\end{equation}
Writing $1 = \chi_0 + 1 - \chi_0$, one gets that 
\begin{align}
& - \lambda(t)  M \, |\xi|^{M - 2} \xi \chi(\xi) (\partial_x g_n)(t, x, \xi) + w_n(t, x, \xi) -  \langle w_n\rangle_x (t, \xi) \nonumber\\
& = - \lambda(t)  M \, |\xi|^{M - 2} \xi \chi(\xi) ( \partial_x g_n )(t, x, \xi) + \chi_0(\xi)\big( w_n(t, x, \xi) -  \langle w_n\rangle_x (t, \xi) \big) \nonumber\\
& \quad + \big(1 - \chi_0(\xi) \big)\big( w_n(t, x, \xi) -  \langle w_n\rangle_x (t, \xi) \big)\,. \label{brotchen 0}
\end{align}
By the definition of $\chi_0$ given in \eqref{scelta cut off chi 0}, one easily gets that  
\begin{equation}\label{resto 1 - chi 0}
\big(1 - \chi_0(\xi) \big)\big( w_n(t, x, \xi) -  \langle w_n\rangle_x (t, \xi) \big) \in S^{- \infty}\,,
\end{equation}
therefore we look for a solution $g_n$ of the equation 
\begin{equation}\label{equazione omologica simbolo g1 b}
- \lambda(t)  M \, |\xi|^{M - 2} \xi \chi(\xi) (\partial_x g_n)(t, x, \xi) + \chi_0(\xi) \big( w_n(t, x, \xi) -  \langle w_n\rangle_x (t, \xi) \big) \in S^{M - (n + 1) \bar{\frak e} }\,. 
\end{equation}
Since we require that ${\cal G}_n = {\rm Op}(g_n)$ is self-adjoint, we look for a symbol of the form 
\begin{equation}\label{forma autoaggiunta g1}
g_n(t, x, \xi) = \sigma_{n}(t, x, \xi) + \sigma_{n}^*(t, x, \xi) \in S^{1 - n  \bar{\frak e}}
\end{equation}
with the property that 
\begin{equation}\label{ansatz g1 equazione omologica}
\sigma_{n}^*(t, x, \xi) = \sigma_{n}(t, x, \xi) + r_{n}(t, x, \xi),\qquad  r_{n} \in S^{- n  \bar{\frak e}}. 
\end{equation}
Plugging the ansatz \eqref{forma autoaggiunta g1} into the equation \eqref{equazione omologica simbolo g1 b}, using \eqref{ansatz g1 equazione omologica} and since  
\begin{equation}\label{brotchen 100}
- \lambda(t)M |\xi|^{M - 2} \xi \chi(\xi) (\partial_x r_{n})(t, x, \xi) \in S^{M - 1 - n  \bar{\frak e}} \subseteq S^{M - (n + 1)  \bar{\frak e}}, 
\end{equation}
we are led to solve the equation
\begin{equation}\label{equazione per sigma n}
- 2 \lambda(t)M |\xi|^{M - 2} \xi \chi(\xi) (\partial_x \sigma_{n})(t, x, \xi) + \chi_0(\xi)\big(w_n(t, x, \xi) - \langle w_n \rangle_x (t, \xi) \big) = 0
\end{equation}
whose solution is given by 
\begin{equation}\label{scelta sigma g1}
\sigma_{n}(t, x, \xi) := 
 \dfrac{\chi_0(\xi) \partial_x^{- 1} \Big[  w_n (t, x, \xi) -  \langle w_n \rangle_x(t, \xi)   \Big]}{2 \lambda(t)M |\xi|^{M - 2} \xi }\,.
 \end{equation}
Since $w_n, \langle w_n\rangle_x \in S^{M - n \bar{\frak e}}$, using that $M > 1$ and recalling the definition of the cut-off function $\chi_0$ in \eqref{scelta cut off chi 0}, one gets that $\sigma_n \in S^{1 - n \bar{\frak e}}$ and hence also $g_n = \sigma_n + \sigma_n^* \in S^{1 - n \bar{\frak e}}$. We now use Lemma \ref{simbolo autoaggiunto fourier multiplier} with $\vphi(t, \xi) = \frac{\chi_0(\xi)}{2 \lambda(t)M |\xi|^{M - 2} \xi}$, $a(t, x, \xi) = \partial_x^{- 1} \Big[ w_n (t, x, \xi)  -  \langle w_n \rangle_x(t, \xi)   \Big]$. Recalling that $w_n = w_n^*$, by Lemma \ref{lemma aggiunto media partial x - 1} we have that $a = a^*$, hence we can apply Lemma \ref{simbolo autoaggiunto fourier multiplier}, obtaining that the ansatz \eqref{ansatz g1 equazione omologica} is satisfied. By \eqref{brotchen 0}, \eqref{forma autoaggiunta g1}, \eqref{ansatz g1 equazione omologica}, \eqref{equazione per sigma n} one then gets  
$$
\begin{aligned}
& - \lambda(t)  M \, |\xi|^{M - 2} \xi \chi(\xi) (\partial_x g_n)(t, x, \xi) + w_n(t, x, \xi) -  \langle w_n\rangle_x (t, \xi)  \\
& =  \big(1 - \chi_0(\xi) \big)\big( w_n(t, x, \xi) -  \langle w_n\rangle_x (t, \xi) \big) -  \lambda(t)M |\xi|^{M - 2} \xi \chi(\xi) (\partial_x r_{n})(t, x, \xi) 
\end{aligned}
$$
and recalling \eqref{resto 1 - chi 0}, \eqref{brotchen 100} one then gets \eqref{equazione omologica simbolo g1}. 
\end{proof}
By Lemmata \ref{prima espansione simbolo v n+1}, \ref{equazione omologica ordini bassi}, the operator ${\cal V}_{n + 1}(t)$ has the form 
\begin{equation}\label{forma finale cal V2 (t)}
{\cal V}_{n + 1}(t) = \lambda(t) |D|^M + \mu_{n + 1}(t, D) + {\cal W}_{n + 1}(t)\,,
\end{equation}
where 
\begin{equation}\label{def mu 1 (t, D)}
\mu_{n + 1}(t, \xi) := \mu_n(t, \xi) + \langle w_n \rangle_x(t, \xi) \in S^{M - \bar{\frak e}}\,,
\end{equation}
\begin{equation}\label{cal W2 (t)}
\begin{aligned}
{\cal W}_{n + 1}(t) & :=  {\rm Op}\Big( r_{v_n}(t, x, \xi)  - \lambda(t)  M \, |\xi|^{M - 2} \xi \chi(\xi) (\partial_x g_n)(t, x, \xi)  \\
& \quad + w_n(t, x, \xi) -  \langle w_n\rangle_x (t, \xi)  \Big) \in OPS^{M - (n + 1)\bar{\frak e}}\,. 
\end{aligned}
\end{equation}
Since by the induction hypothesis $w_n = w_n^*$ and $\mu_n(t, \xi)$ is real, by \eqref{aggiunto Fourier multiplier} and Lemma \ref{lemma aggiunto media partial x - 1}-$(i)$, one has that $\langle w_n \rangle_x (t, \xi)$ is real and therefore $\mu_{n + 1}(t, \xi)$ is real. Furthermore, since $\Phi_n$ is symplectic and $\ii {\cal V}_n$ is a Hamiltonian vector field, one has that $\ii {\cal V}_{n + 1}$ is still a Hamiltonian vector field, meaning that ${\cal V}_{n + 1}$ is self-adjoint. Using that $\mu_{n + 1}(t, \xi)$ is a real Fourier multiplier, one has that $\lambda(t) |D|^M + \mu_{n + 1}(t, D)$ is a self-adjoint operator, implying that 
$$
{\cal W}_{n + 1}(t) = {\cal V}_{n + 1}(t) - \lambda(t) |D|^M - \mu_{n + 1}(t, D)
$$
is self-adjoint too. Then, the proof of Proposition \ref{descent method M geq 1} is concluded. 
\section{Proof of Theorem \ref{teorema riduzione}}\label{prova teorema riduzione conclusa}
Let $K \in \N$ and let us fix a positive integer $N_K \in \N$ as 
\begin{equation}\label{teorema principale K}
N_K := \Big[ \frac{M + K}{\bar{\frak e}} \Big] + 1
\end{equation}
so that $M - N_K \bar{\frak e} < - K$ (for any $x \in \R$, we denote by $[x]$ its integer part). Then we define
\begin{equation}\label{trasformazione mappa teo principale}
\begin{aligned}
& {\cal T}_K(t) : = \Phi(t)^{- 1} \circ \Phi_1(t)^{- 1} \circ \ldots \circ \Phi_{N_K - 1}(t)^{- 1}\,, \\
&  {\cal W}_K(t) := {\cal W}_{N_K}(t)\,, \quad \lambda_K(t, D) := \lambda(t) |D|^M + \mu_{N_K}(t, D)
\end{aligned}
\end{equation}
where $\Phi(t) = \Phi(1; t)$ is given by \eqref{equazione flusso operatoriale}, $\lambda(t)$ is defined in \eqref{lambda 1 (t)} and for any $n \in \{1, \ldots, N_K - 1\}$, $\Phi_n(t)$, ${\cal W}_n(t)$, $\mu_n(t, D)$ are given in Theorem \ref{descent method M geq 1}. By  \eqref{proprieta Phi 2}, \eqref{proprieta Phi n teo 1}, using the product rule, one gets that ${\cal T}_K$ satisfies the property \eqref{proprieta cal TK 2}. Furthermore, by \eqref{push forward cal V 1 teorema}, \eqref{forma cal Vn teorema}, \eqref{cal V n + 1 cal V n} one obtains \eqref{campo finalissimo cal VK}, with $\lambda_K(t, D)$, ${\cal W}_K(t)$ defined in \eqref{trasformazione mappa teo principale}, hence the proof of Theorem \ref{teorema riduzione} is concluded. 
\section{Proof of Theorem \ref{teo growth of sobolev norms}.}\label{prova teorema finalissimo}
Let $s > 0$, $t_0 \in \R$, $u_0 \in H^s(\T)$. We fix the constant $K \in \N$, appearing in Theorem \ref{teorema riduzione}, as
\begin{equation}\label{relazione K s}
K = K_s :=[s] + 1
\end{equation}
so that $K > s$.  
By applying Theorem \ref{teorema riduzione}, one has that $u(t)$ is a solution of the Cauchy problem 
\begin{equation}\label{cauchy problem prova main theorem}
\begin{cases}
\partial_t u + \ii {\cal V}(t)[u] = 0 \\
u(t_0) = u_0
\end{cases}
\end{equation}
if and only if $v(t) := {\cal T}_{K_s}^{- 1}(t) u(t)$ is a solution of the Cauchy problem 
\begin{equation}\label{cauchy problem trasformato main theorem}
\begin{cases}
\partial_t v + \ii \lambda_{K_s}(t, D) v + \ii {\cal W}_{K_s}(t)[v] = 0 \\
v(t_0) = v_0\,,
\end{cases} \qquad v_0 := {\cal T}_{K_s}^{- 1}(t_0) [u_0]\,
\end{equation}
with $\lambda_{K_s}(t, D) = {\rm Op}(\lambda_{K_s}(t, \xi)) \in OPS^M$ with $\lambda_{K_s}(t, \xi) = \overline{\lambda_{K_s}(t, \xi)}$. Since the symbol $\lambda_{K_s}(t, \xi)$ is real, we have 
\begin{equation}\label{lambda Ks autoaggiunto}
\lambda_{K_s}(t, D) = \lambda_{K_s}(t, D)^*\,.
\end{equation}  
Moreover, since $K_s > s > 0$, by \eqref{inclusioni Sm OPSm}, one has 
\begin{equation}\label{proprieta cal W Ks teorema}
{\cal W}_{K_s}(t) \in OPS^{- K_s} \subset OPS^{- s} \subset OPS^0\,, \quad {\cal W}_{K_s}(t) = {\cal W}_{K_s}(t)^*\,. 
\end{equation}
By applying Lemma \ref{cauchy schrodinger astratto} one gets that there exists a unique global solution $v \in {\cal C}^0(\R, H^s)$ of the Cauchy problem \eqref{cauchy problem trasformato main theorem}, therefore $u \in {\cal C}^0(\R, H^s)$ is the unique solution of the Cauchy problem \eqref{cauchy problem prova main theorem}. In order to conclude the proof, it remains only to prove the bound \eqref{scopo dell articolo 0}. 

\medskip

\noindent
{\sc Estimate of $v(t)$.} By a standard energy estimate, using \eqref{lambda Ks autoaggiunto}, \eqref{proprieta cal W Ks teorema}, one gets easily that 
\begin{equation}\label{norma L2 v conservata}
\| v(t) \|_{L^2} = \| v_0\|_{L^2}\,, \qquad \forall t \in \R\,. 
\end{equation}
Writing the Duhamel formula for the Cauchy problem \eqref{cauchy problem trasformato main theorem}, one obtains 
\begin{equation}\label{duhamel equazione ridotta}
v(t) = e^{\ii \Lambda_{K_s}(t, D)} v_0 + \int_{t_0}^t e^{\ii \big(\Lambda_{K_s}(t, D) - \Lambda_{K_s}(\tau, D) \big)} {\cal W}_{K_s}(\tau)[v(\tau)]\, d \tau
\end{equation} 
where 
$$
\Lambda_{K_s}(t, D) := {\rm Op}\Big( \Lambda_{K_s}(t, \xi)\Big)\,, \qquad \Lambda_{K_s}(t, \xi) := \int_{t_0}^t \lambda_{K_s}(\tau, \xi)\, d \tau\,. 
$$
Since $\lambda_{K_s}(t, \xi)$ is real, $\Lambda_{K_s}(t, \xi)$ is real too, and therefore the propagator $e^{\ii \Lambda_{K_s}(t, D)}$ is unitary on $H^s(\T)$. Hence, one has 
\begin{align}
\| v(t) \|_{H^s} & \leq \| v_0\|_{H^s} + \Big| \int_{t_0}^t \| {\cal W}_{K_s}(\tau)[v(\tau)] \|_{H^s}\, d \tau  \Big| \nonumber\\
& \stackrel{\eqref{proprieta cal W Ks teorema}, \,Theorem \,\ref{conitnuita pseudo}\,,}{\lesssim_s} \| v_0\|_{H^s} + \Big|\int_{t_0}^t \| v(\tau)\|_{L^2}\, d \tau \Big| \nonumber\\
& \stackrel{\eqref{norma L2 v conservata}}{\lesssim_s} \| v_0\|_{H^s} +  |t - t_0| \| v_0\|_{L^2}\,.  \label{stima teorema principale semi-finale}
\end{align}

\medskip

\noindent
{\sc Estimate of $u(t)$.} Since $u(t) = {\cal T}_{K_s}(t)[v(t)]$ and $v_0 = {\cal T}_{K_s}(t_0)^{- 1}[u_0]$ and ${\cal T}_{K_s}(t)^{\pm 1}$ satisfy \eqref{proprieta cal TK 2}, one gets that 
$$
\| u(t)\|_{H^s} \simeq_s \| v(t)\|_{H^s}\,, \quad \| u_0\|_{H^s} \simeq_s \| v_0\|_{H^s}\,, \quad  \| v_0\|_{L^2} \simeq \| u_0\|_{L^2}\,,
$$
therefore, by \eqref{stima teorema principale semi-finale} one deduce that 
\begin{equation}\label{stima lineare t u(t)}
\| u(t) \|_{H^s} \lesssim_s \| u_0 \|_{H^s} + |t - t_0| \| u_0\|_{L^2}\,. 
\end{equation} 

\medskip

\noindent
{\sc Proof of \eqref{scopo dell articolo 0}.} The estimate \eqref{stima lineare t u(t)} proves that the propagator ${\cal U}(t_0, t)$ of the PDE $\partial_t u = \ii {\cal V}(t)[u]$, i.e. 
$$
\begin{cases}
\partial_t {\cal U}(t_0, t) = \ii {\cal V}(t) {\cal U}(t_0, t) \\
{\cal U}(t_0, t_0) = {\rm Id}
\end{cases}
$$ 
satisfies 
\begin{equation}\label{propagatore cal U Hs}
\| {\cal U}(t_0, t)\|_{{\cal B}(H^s)} \lesssim_s 1 + |t - t_0|\,, \quad \forall s > 0\,, \quad \forall t, t_0 \in \R\,. 
\end{equation}
 Furthermore, since ${\cal V}(t)$ is self-adjoint, the $L^2$ of the solutions is constant, namely 
\begin{equation}\label{propagatore cal U L2}
\| {\cal U}(t_0, t)\|_{{\cal B}(L^2)}= 1\,, \quad  \forall t, t_0 \in \R\,.
\end{equation}
Hence, for any $0 < s < S$, by applying Theorem \ref{interpolazione sobolev}, one gets that 
\begin{align}
\| {\cal U}(t_0, t)\|_{{\cal B}(H^s)} & \leq \| {\cal U}(t_0, t)\|_{{\cal B}(L^2)}^{\frac{S - s}{S}} \| {\cal U}(t_0, t)\|_{{\cal B}(H^S)}^{\frac{s}{S}} \stackrel{\eqref{propagatore cal U Hs}, \eqref{propagatore cal U L2}}{\lesssim_S} (1 + |t - t_0|)^{\frac{s}{S}}\,.
\end{align}
Then, for any $\e > 0$, choosing $S$ large enough so that $s/S \leq \e$, the estimate \eqref{scopo dell articolo 0} follows. This concludes the proof of Theorem \ref{teo growth of sobolev norms}.


\smallskip

\begin{flushright}
Riccardo Montalto,  University of Z\"urich, Winterthurerstrasse 190,
CH-8057, Z\"urich, Switzerland. \\ \emph{E-mail: {\tt riccardo.montalto@math.uzh.ch}} 
\end{flushright}

\end{document}